\documentclass[pdflatex,sn-mathphys-num]{sn-jnl}

\usepackage{graphicx}%
\usepackage{multirow}%
\usepackage{amsmath,amssymb,amsfonts}%
\usepackage{amsthm}%
\usepackage{mathrsfs}%
\usepackage[title]{appendix}%
\usepackage{xcolor}%
\usepackage{textcomp}%
\usepackage{manyfoot}%
\usepackage{booktabs}%
\usepackage{algorithm}%
\usepackage{algorithmicx}%
\usepackage{algpseudocode}%
\usepackage{listings}%
\usepackage{commath}

\usepackage{mathtools}

\usepackage{geometry}
\usepackage{fullpage}
\usepackage{tikz}
\usetikzlibrary{graphs}
\usetikzlibrary{positioning}

\usepackage{hyperref}
\usepackage{multirow}
\usepackage{xcolor}
\usepackage{breqn}
\usetikzlibrary{matrix}
\input colordvi
\renewcommand\a{\alpha}
\renewcommand\b{\beta}
\newcommand\g{\gamma}
\renewcommand\d{\delta}
\newcommand\e{\varepsilon}
\renewcommand\l{\lambda}

\newcommand\D{\Delta}

\newcommand\f{\frac}

\newcommand{\Z}{{\mathbb{Z}}}
\newcommand{\R}{{\mathbb{R}}}

\newcommand{\C}{{\mathbb{C}}}

\renewcommand\i{^{-1}}
\renewcommand\({\left(}
\renewcommand\){\right)}

\newcommand{\sgn}{\operatorname{sgn}}

\newcommand{\wtilde}{\widetilde{w}}
\newcommand{\ntilde}{\widetilde{n}}
\newcommand{\utilde}{\widetilde{u}}
\newcommand{\atilde}{\widetilde{a}}
\newcommand{\btilde}{\widetilde{b}}

\newcommand{\gobble}[1]{}
  \newcommand{\rangeref}[2]{%
    \ref{#1}--\afterassignment\gobble\fam 0\ref{#2}%
  }
\newcommand{\crarrow}{\textsubscript{\ooalign{$\swarrow$\cr\hfil$\nwarrow$\hfil}}}
\newenvironment{smallarray}[1]
 {\null\,\vcenter\bgroup\scriptsize
  \arraycolsep=.13885em
  \hbox\bgroup$\array{@{}#1@{}}}
 {\endarray$\egroup\egroup\,\null}

\newcommand{\GL}{\operatorname{GL}}

\newcommand{\djlz}{D_{j,l}^{(0)}}

\newcommand{\djlktwo}{D_{j,l,k}^{(2)}}
\newcommand{\ojlktwo}{O_{j,l,k}^{(2)}}
\newcommand{\ejlktwo}{\e_{j,l,k}^{(2)}}
\newcommand{\npdj}{n_{\pi(d),j}}
\newcommand{\npdl}{n_{\pi(d),l}}
\newcommand{\nprj}{n_{\pi(r),j}}

\newcommand{\eqand}{\quad\text{and}\quad}

\makeatletter
\def\cline#1{\@cline#1\@nil}
\def\@cline#1-#2\@nil{%
  \omit
  \@multicnt#1%
  \advance\@multispan\m@ne
  \ifnum\@multicnt=\@ne\@firstofone{&\omit}\fi
  \@multicnt#2%
  \advance\@multicnt-#1%
  \advance\@multispan\@ne
  \leaders\hrule\@height\arrayrulewidth\hfill
  \cr
  \noalign{\vskip-\arrayrulewidth}}
\makeatother

\numberwithin{equation}{section}

\newtheorem{thm}[equation]{Theorem}
\newtheorem{cor}[equation]{Corollary}
\newtheorem{lem}[equation]{Lemma}

\newtheorem{prop}[equation]{Proposition}
\newtheorem{definition}[equation]{Definition}

\begin{document}

\title[Article Title]{Schubert cells and Whittaker functionals for $\GL(r,\R)$ part I: Combinatorics}


\author*[1]{\fnm{Doyon} \sur{Kim}}\email{kimdoyon@math.uni-bonn.de}

\affil*[1]{\orgdiv{Mathematisches Institut}, \orgname{Universit\"{a}t Bonn}, \orgaddress{\street{Endenicher Allee 60}, \postcode{53129} \city{Bonn}, \country{Germany}}}


\abstract{We give a formula for a birational map on the Schubert cell associated to each Weyl group element of $G=\GL(r)$.
The map simplifies the UDL decomposition of matrices, providing structural insight into the Schubert cell decomposition of the flag variety $G/B$, where $B$ is a Borel subgroup. An application of the formula includes a new proof of the existence of Whittaker functionals for principal series representations of $\GL(r,\R)$ via integration by parts. In this paper, we establish combinatorial properties of the birational map and prove auxiliary results.}




\maketitle

\section{Introduction}\label{sec:intro}
Whittaker functions are of crucial importance in studying automorphic forms, because they appear as the Fourier coefficients of automorphic forms and hence are used to construct the associated local $L$-functions. In terms of the representation theory of real reductive Lie groups $G$, they are completely described by Whittaker functionals, that is, the continuous linear functionals $\tau : V^\infty \to \C$ on the space of smooth vectors in an irreducible representation $(\pi,V)$ of $G$ which satisfy
\begin{equation}
\label{whitfunl}
    \tau\left(\pi(n)v\right) = \psi(n)\,\tau(v) \quad \text{for all} \quad n\in N, 
\end{equation}
where $N$ is the unipotent radical of a minimal parabolic subgroup of $G$ and $\psi : N \to \C^*$ is a nondegenerate character. The particular choices of $N$ and $\psi$ do not matter since $N$ is unique up to $G$-conjugacy and $\psi$ up to conjugacy under the normalizer of $N$. The ``multiplicity one theorem" asserts that the space of Whittaker functionals on irreducible representations of $\GL(r,\R)$ is at most one-dimensional. The multiplicity one theorem was originally proven by Piatetski-Shapiro \cite{piatetski1979multiplicity} and Shalika \cite{shalika1974multiplicity}. In \cite{kostant1978whittaker}, Kostant showed that the dimension of the space of Whittaker functionals for any principal series representation of a quasisplit linear Lie group is exactly one. \par
Via the Casselman-Wallach embedding, any irreducible representation of $\GL(r,\R)$ can be embedded into a principal series representation (see \cite{casselman1989canonical} and \cite{wallach2006asymptotic}). The Whittaker functionals for a principal series representation $(\pi,V)$ can be realized as distribution vectors satisfying the transformation law \eqref{whitfunl}, and the transformation law uniquely determines the Whittaker functionals on the open Schubert cell. The existence and uniqueness of the Whittaker functional $\tau$ amounts to the assertion that $\tau$ admits a canonical extension to the whole group $G$, and that the extension becomes unique if it is required to transform under the action of $N$ according to the character $\psi$. Casselman, Hecht, and Mili\v{c}i\'{c} take this point of view in \cite{casselman2000bruhat}, where they establish the existence and uniqueness of Whittaker functionals. \par
This paper is the first in a series of two papers that together give a new proof of the existence of Whittaker models on principal series representations of $\GL(r,\R)$ by reducing the analysis of Whittaker functionals to integration by parts. This series is based on the author's Ph.D. thesis \cite{kimWhit}, which originally presented this new proof. The proofs have been modified and improved in many places for clarity, but this paper refers to \cite{kimWhit} for more details where appropriate.
\subsection{Set-up and definitions}
Let $F$ be a field of characteristic $0$. Let $G=\GL(r,F)$, $B_{-}$ be its lower triangular Borel subgroup, and $N$ (resp.\ $N_{-}$) is the group of upper triangular (resp.\ lower triangular) unipotent matrices. The group $N$ can be identified with $F^d$ with $d=\frac{1}{2}r(r-1)$ via the coordinate entries $n_{i,j}$, $1\le i<j\le r$ of $n\in N$. Let $\Omega$ be the Weyl group of $G$. The Weyl group consists of permutation matrices $w=(w_{i,j})$ with 
\begin{equation}\label{pidef}w_{i,j}=\d_{j=\pi(i)},\end{equation}
in bijection with permutations $\pi$ in the symmetric group $\mathfrak{S}_{r}$. The group $G$ has the Bruhat decomposition $ G=\bigcup_{w\in\Omega} NwB_{-}$, where the union is disjoint. For $w\in \Omega$, the $N$-orbit $NwB_{-}$ of $wB_{-}$ is called the \emph{Schubert cell attached to} $w$. The Schubert cell attached to the identity element of the Weyl group is also called the \emph{open Schubert cell}. We write $S_w=NwB_{-}$. For each $w\in\Omega$ we may decompose $N$ as the product of the subgroups
\begin{equation}\label{nws}
    \aligned
    N^w & = \{n\in N\,|\,w\i n w\in N_{-}\} =  N\cap wN_{-}w\i,\\
    N_w & = \{n\in N\,|\,w\i n w\in N\} =   N \cap wNw\i = 
    w N_{w\i}w\i.
    \endaligned
\end{equation}
It is clear that for each $w\in\Omega$, we have $N=N_wN^w  = N^wN_w$.
\begin{lem}
We have $S_w=N_w w B_{-}$, hence $S_w$ is a left $N_w$-orbit of $wB_{-}$. The Schubert cell $S_w$ is contained in the set $wNB_{-}$, which satisfies $wNB_{-}= N_w w N^{w\i}B_{-}$.
\end{lem}
\begin{proof}
We have
$S_w=\left(N_wN^w\right)wB_{-}=N_w\left(N\cap wN_{-}w\i\right)wB_{-}=N_w wB_{-}$. The Schubert cell $S_w$ further satisfies
\[
  S_w=N_w wB_{-}=wN_{w\i}B_{-}\subseteq wNB_{-}=wN_{w\i}N^{w\i}B_{-}=N_w w N^{w\i}B_{-},\]
thus $S_w$ is contained in $wNB_{-}$, the left $N_w$-orbit of $wN^{w\i}B_{-}$.
\end{proof}
Write $Y_w=wNB_{-}$. Since $Y_w$ contains $S_w$, we can write $G$ as an overlapping union
\begin{equation}\label{overlappingunion}
G=\bigcup\limits_{w\in\Omega}Y_w=\bigcup\limits_{w\in\Omega}N_wwN^{w\i}B_{-}.\end{equation}
Each set $Y_w=N_wwN^{w\i}B_{-}$ has a large overlap with the open Schubert cell $NB_{-}$. The overlap is dense and open in $Y_w$. Writing $g\in Y_w$ as $g=n_1\cdot w\cdot n\cdot b$ with $n_1 \in N_w$, $n\in N^{w\i}$, and $b\in B_{-}$, we see that $g$ lies in the open Schubert cell if and only if the matrix $wn\in wN^{w\i}$ admits a $UDL$ decomposition. Henceforth we shall assign a set of coordinates on $wN^{w\i}B_-$ that describes the entries of $wn$. \par 
Let $e_{i,j}$ be an elementary matrix with $1$ in the $(i,j)$-entry and zeros in all other entries. Then a Weyl group element $w\in \Omega$ and the corresponding permutation $\pi\in \mathfrak{S}_{r}$ satisfy $w e_{i,j} w\i=e_{\pi\i(i),\pi\i(j)}$.
By \eqref{nws}, the entries $n_{i,j}$ of $n=(n_{i,j})\in N^{w\i}$ above the main diagonal are free if $\pi\i(i)>\pi\i(j)$, and zero otherwise. Let
\begin{equation}\label{invpiinv}
    Inv(\pi\i)  =  \{ (i,j)\mid i<j, \ \pi\i(i) > \pi\i(j) \}
\end{equation}
denote the set of indices of free entries in $n\in N^{w\i}$. If a matrix $X$ has entries $x_{i,j}$, then the entries of $wX$ are
    $(wX)_{i,j} = \sum_{k=1}^n w_{i,k}x_{k,j} = x_{\pi(i),j}$. Accordingly, we denote the entries of $wn$ as
\begin{equation}\label{wnentries}
(wn)_{i,j}=n_{\pi(i),j}, \quad \text{or equivalently,} \quad n_{i,j}=(wn)_{\pi\i(i),j}.
\end{equation}
By (\ref{invpiinv}), the entries $n_{i,j}$ of $wn$ are free if $(i,j)\in Inv(\pi\i)$. We call these entries the \emph{free variables of} $wn$. By (\ref{pidef}), $n_{i,j}=1$ if $i=j$. Finally, $n_{i,j}=0$ if $i>j$ or $\pi\i(i)<\pi\i(j)$. We denote the set of free variables as $V_w$, so
\begin{equation}\label{coords}
    V_w=\{n_{i,j}\,|\,(i,j)\in Inv(\pi\i)\}.
\end{equation}
For example, consider $w=\left(
\begin{smallarray}{cccc}
 0 & 1 & 0 & 0 \\
 0 & 0 & 0 & 1 \\
 0 & 0 & 1 & 0 \\
 1 & 0 & 0 & 0 \\
\end{smallarray}
\right)$ with corresponding permutation $\pi(1)=2$, $\pi(2)=4$, $\pi(3)=3$, and $\pi(4)=1$. We have
$Inv(\pi\i)=\left\{(1,2),(1,3),(1,4),(3,4)\right\}$, and the general form of matrices in the set $N^{w\i}$ and $wN^{w\i}$ are
\begin{equation} \label{ex0}
n=\left(
\begin{smallarray}{cccc}
 1 & n_{1,2} & n_{1,3} & n_{1,4} \\
 0 & 1 & 0 & 0 \\
 0 & 0 & 1 & n_{3,4} \\
 0 & 0 & 0 & 1 \\
\end{smallarray}
\right) \quad\text{and}\quad wn=\left(
\begin{smallarray}{cccc}
 0 & 1 & 0 & 0 \\
 0 & 0 & 0 & 1 \\
 0 & 0 & 1 & n_{3,4} \\
 1 & n_{1,2} & n_{1,3} & n_{1,4} \\
\end{smallarray}
\right),\end{equation}
respectively. We often use the notation $\a=(i,j)$ for $(i,j)\in Inv(\pi\i)$.
\subsection{Main result of this series} A matrix $g\in \GL(r)$ fails to have a $UDL$ decomposition if any of the lower right subblocks of $g$ has determinant $0$. For example, matrix $wn$ in \eqref{ex0} is not in $NB_{-}$ if $n_{1,4}=n_{1,3}n_{3,4}$ or $n_{1,2}=0$. This criterion becomes more and more complicated as $r$ increases. Instead, consider the following change of coordinates: $R(n_{1,2})=n_{1,2}n_{1,3}n_{1,4}$, $R(n_{1,3})=\frac{n_{1,3} n_{1,4}+n_{3,4} n_{1,4}}{n_{3,4}}$, $R(n_{1,4})=n_{1,4}n_{3,4}$, $R(n_{3,4})=n_{3,4}$. We write $u_{i,j}=R(n_{i,j})$ for $n_{i,j}\in V_w$. Under the map $R$, the matrix $wn$ transforms to
\[wu=\left(
\begin{smallarray}{cccc}
 0 & 1 & 0 & 0 \\
 0 & 0 & 0 & 1 \\
 0 & 0 & 1 & n_{3,4} \\
 1 & n_{1,2} n_{1,3} n_{1,4} & \frac{n_{1,3} n_{1,4}+n_{3,4} n_{1,4}}{n_{3,4}} & n_{1,4} n_{3,4} \\
\end{smallarray}
\right).\]
We have $wu=x\cdot b$,
where
\[x=\left(
\begin{smallarray}{cccc}
 1 & \frac{1}{n_{1,2}} &  0& 0 \\
  & 1 & \frac{1}{n_{3,4}}+\frac{1}{n_{1,3}} & \frac{1}{n_{1,4} n_{3,4}} \\
  &  & 1 & \frac{1}{n_{1,4}} \\
  &  &  & 1 \\
\end{smallarray}
\right),\quad b=\left(
\begin{smallarray}{cccc}
 -\frac{1}{n_{1,2} n_{1,3} n_{1,4}} &  & &  \\
\star & n_{1,2} &  &  \\
 \star  & \star & -\frac{n_{1,3}}{n_{3,4}} &  \\
 \star  & \star  & \star  & n_{1,4} n_{3,4} \\
\end{smallarray}
\right).\]
Observe that the superdiagonal entries of $x$ are a reciprocal sum of variables in $V_w$, and the diagonal entries of $b$ are monomials in $\Z[\{n_\a,n_\a\i\mid n_\a\in V_w\}]$. From this, we immediately see that $wu$ lies in $NB_-$ precisely when every $n_\a$ is nonzero. Our main theorem states that such a birational map $R=R_w$ exists for each $w\in \Omega$ of $\GL(r)$ with $r\geq 2$.
\begin{thm}\label{biratlthm}
For each $w\in\Omega$, there exists a birational map $R=R_w$
\[V_w \simeq F^{\, \abs{V_w}} \ \stackrel{R}{ \longrightarrow} \{u_{i,j}\,|\,n_{i,j}\in V_w\}\simeq F^{\,\abs{V_w}} \]
which satisfies the following properties:
\begin{enumerate}
\item[(i)] $R$ is smooth, of maximal rank, on $(F^\times)^{\abs{V_w}}$.
\item[(ii)] Let $wu$ denote the matrix obtained from $wn$ by replacing $n_\a\in V_w$ with $u_\a$. 
If each $n_{\alpha}\neq 0$, then $wu$ decomposes as $x\cdot b$, where $x
\in N$ and $b\in B_{-}$. The $i$-th superdiagonal entry $x_{i,i+1}$ of $x$ is given by
$x_{i,i+1}=\sum\limits_{n_{\alpha}\in B_w(i+1)} \frac{1}{n_{\alpha}}$, where $B_w(i)\subseteq V_w$ with $2\leq i\leq r$ partition $V_w$. Thus we have
$\sum\limits_{i=1}^{r-1} x_{i,i+1}=\sum\limits_{n_\alpha \in V_w} \frac{1}{n_\alpha}$.
\item[(iii)] The $i$-th diagonal entry $b_{i,i}$ of $b$ is given by
\begin{equation}\label{bdiag}
    b_{i,i}=(-1)^{i+\pi\i(i)}\cdot\frac{\prod\limits_{(a,i)\in Inv(\pi\i)}(-n_{a,i})}{\prod\limits_{(i,b)\in Inv(\pi\i)}n_{i,b}}.
\end{equation}
  \item[(iv)] The gauge form $\prod\limits_{(i,j)\in Inv(\pi\i)}du_{i,j}$
    transforms to $ \prod\limits_{(i,j)\in Inv(\pi\i)} \lvert n_{i,j}\rvert^{j-i-1}dn_{i,j}$.
  \item[(v)] There exists an ordering $\prec$ of the indices $\alpha \in Inv(\pi\i)$ such that for any pair of indices $\a$ and $\b$, $n_\a\f{\partial }{\partial n_\a} u_\b$ has the form \begin{equation}\label{ordform}
 \sum\limits_{i} c_i(\a)f_i, \quad \text{where}\quad c_i(\a)\in \Z\bigl[\{n_\d,n_\d\i\mid \d\prec\a\}\bigr] \quad\text{and}\quad f_i\in \Z\bigl[\{u_\g \mid n_\g \in V_w \} \bigr].
\end{equation}
\end{enumerate}
\end{thm}
The definition of $B_w(i)$ and the ordering $\prec$ will be provided in the second paper of this series. The map $R$ is defined in a combinatorial manner, so the proof of Theorem~\ref{biratlthm} relies heavily on combinatorial techniques. The present paper gives an explicit formula for the map $R$ and establishes preparatory results via a combinatorial analysis of the map. The second paper of this series will prove parts (i) through (v) using the results we provide here. See Section~\ref{sec:paper1result} for the results in this paper.
\par
In their unpublished note \cite{miller2008unpublished}, Miller and Schmid obtained an explicit formula for $R_{w_l}$ that satisfies parts (i) through (v), where $w_l$ is the longest Weyl group element of $\GL(r)$ for any $r\geq 2$. Also, they provided an algorithm that computes $R_w$ for a given Weyl group element (see Appendix~\ref{sec:algorithm} for the algorithm) and conjectured that the formula obtained by the algorithm satisfies Theorem~\ref{biratlthm}. Using Mathematica, it is verified that their formula for $R_w$ for each Weyl group element $w$ of $\GL(r)$ with $2\leq r\leq 7$ satisfies parts (ii) and (iii). However, they were unable to prove general statements using the algorithmic approach. 
\subsection{Application: Whittaker distributions on $\GL(r,\R)$} Let $G=\GL(r,\R)$, and let $(\pi,V)$ be a principal series representation induced by the character $\chi$ on $B_{-}$ defined by $\chi(b)=\prod_{i=1}^{r} \abs{b_{i,i}}^{t_i}\sgn(b_{i,i})^{\varepsilon_i}$, where $t_i\in \C$ and $\varepsilon_i\in \{0,1\}$. Let $e(z)=e^{2\pi i z}$, and let $\psi(n)=e(n_{1,2}+\cdots+n_{r-1,r})$ be the standard Whittaker character on $N$. Whittaker distributions on $G$ are characterized by the transformation law $\tau(ngb)=\psi(n)\tau(g)\chi(b)$ where $n\in N$, $g\in G$, and $b\in B_{-}$. The transformation law determines the Whittaker distribution $\tau$ on the open Schubert cell $NB_{-}$ uniquely up to constant. \par
In \cite{miller2004distributions}, Miller and Schmid introduced the notion of \emph{distributions vanishing to infinite order}, and showed that if a distribution $\sigma$ on $M$ vanishes to infinite order along a submanifold $S$, then $\sigma$ is uniquely determined by its restriction to $M\smallsetminus S$ \cite[Lemma 2.8]{miller2004distributions}. A prototypical example is the function $e(1/x)$. As a distribution on $\R\smallsetminus \{0\}$, $e(1/x)$ vanishes to infinite order at $0$ (see \cite[Lemma 1.23]{kimWhit}), hence extends to a distribution on $\R$. Parts (ii) and (iii) of Theorem~\ref{biratlthm} imply that we have
\begin{equation}\label{tauwu}\tau(wu)=\pm e\Bigl(\sum_{\a\in Inv(\pi\i)}\frac{1}{n_\a}\Bigr)\prod_{\a\in Inv(\pi\i)} \abs{n_{\a}}^{s_\a}\sgn (n_\a)^{\eta_\a}\end{equation}
where $s_\alpha\in \C$, $\eta_\alpha\in \{0,1\}$, and the sign $\pm$ depends on the parameters for $\chi$ and $w$. Via a multivariable analogue of the argument for $e(1/x)$, which will be provided in detail in the second paper of this series, Theorem~\ref{biratlthm} implies that $\tau$ vanishes to infinite order along the complement of $NB_{-}$ in $G$, therefore admits a canonical extension to $G$. This proves the existence of Whittaker functionals. This technique involves pairings of distributions and
integration by parts. This technique was used in \cite{miller2012archimedean} by Miller and Schmid to prove the analytic continuation of the exterior square $L$-functions, and also in \cite{kim2023infinitely} by the author to prove the infinitude of the critical zeros of $L$-functions with additive twists. See \cite{miller2013mathematics} for an overview of this method.
\subsection{Results of this paper}\label{sec:paper1result}
The map $R$ is defined in a combinatorial manner. We regard the general matrix $wn\in wN^{w\i}$ as a directed graph, and define $R$ using nonintersecting path tuples within the graph. In this paper, we use combinatorial techniques to establish preparatory results necessary for proving Theorem~\ref{biratlthm}. Our analysis is closely related to the Lindstr\"{o}m-Gessel-Viennot lemma \cite{gessel1985binomial,lindstrom1973vector}, which computes the determinant of the weighted path matrix of an acyclic graph. Our main combinatorial lemma (Lemma~\ref{lem:firststep}) uses the ``tail-swapping" technique: if two paths intersect, swapping their tails from the vertex of intersection simplifies the description of the entanglement. This is also the key step in the proof of the Lindstr\"{o}m-Gessel-Viennot lemma. \par 
The following lemma provides explicit formulas for the $UDL$ decomposition $g=nhn_{-}$ of a general matrix $g$, in particular, the superdiagonal entries of $n$ and the entries of the $h$.
\begin{lem}[Corollary~4.3 of \cite{miller2012archimedean}]\label{UDLlem}
If $g =n h n_-$, with $n$, $h$, $n_-$ upper triangular unipotent,
diagonal, and lower triangular unipotent, respectively, then
    \begin{align}
&h_{i,i}  = 
\f{\det\bigl((g_{k,\ell})_{\,k,\ell\geq
i}\bigr)}{\det\bigl((g_{k,\ell})_{\,k,\ell>i}\bigr)} \quad \text{and} \quad n_{i,i+1} = \f{\det\Bigl((g_{k,\ell})_{\stackrel{\scriptstyle{k\geq i,\, k\neq
i+1}}{\ell>i }}\Bigr)}{\det\bigl((g_{k,\ell})_{k,\ell>i}\bigr)}\ . \label{superdiag}
    \end{align}
\end{lem}
This lemma reduces parts (ii) and (iii) to computing the determinants of certain lower right subblocks of $wu$. We translate the combinatorial properties of $R$ into a linear algebraic framework to prove the following, which is the main result of this paper.
\begin{prop}\label{thm:rowcolop}
Let $T_i$ and $F_i$ be the square submatrix of $wu$ of size $i$ defined by \[T_i=\left((wu)_{k.\ell}\right)_{k,\ell\geq r-i+1}, \quad F_i=\left((wu)_{k,\ell}\right)_{\stackrel{\scriptstyle{k\geq r-i,\, k\neq
r-i+1}}{\ell\geq r-i+1 }},\]
respectively. Let $T_i^L$ and $F_i^L$ be the matrices of size $(i-1)\times (i-1)$ defined in Definition~\ref{TLFLdef}. Assume that $1\leq i\leq r$ satisfy $r-i+1>\pi(r)$ so that the bottom row of $T_i$ and $F_i$ does not contain the $1$-entry in the bottom row of $wu$, which is in the $\pi(r)$-th column.
    \begin{enumerate}
        \item[(i)] We have
$\det (T_i)=(-1)^{s+i} \cdot u_{\pi(r-i+s),r}\cdot \det( T_i^L)$, where $s$ is such that the $(r-i+s)$-th row of $wn$ coincides with the row of $T_i$ that contains the highest nonzero entry of the rightmost column of $T_i$.
\item[(ii)] We have
\[\det (F_i)=\begin{cases}
    \frac{1}{n_{\pi(r-i+1),r}}\det(T_i)-\frac{(-1)^{i}}{n_{\pi(r-i+1),r}}\cdot u_{\pi(r-i+2),r}\cdot \det (F^L_i) & \text{ if } r-i+1>\pi\i(r) \\
   (-1)^{i}\cdot u_{\pi(r-i+2),r}\cdot \det (F^L_i) & \text{ if } r-i+1=\pi\i(r) \\
   (-1)^{r-\pi\i(r)}\cdot \det (F^L_i) & \text{ if } r-i+1<\pi\i(r).
\end{cases}\]
    \end{enumerate}
\end{prop}
As we shall see in later sections, the matrices $T_i^L$ and $F_i^L$ contain the information of a Weyl group element of size $(r-1)$ rather than $r$. We will use this proposition to prove parts (ii) and (iii) by induction. \par 
In Appendix~\ref{sec:algorithm}, we outline two different algorithms that compute $R$. The first is the algorithm by Miller \cite{miller2008unpublished}, and the second follows the combinatorial definition of $R$ defined in this paper.
Using Mathematica, the author verified that the two algorithms agree for every Weyl group element of size up to $7$.
\section*{Acknowledgements}
This problem started from the unfinished work of Stephen D. Miller and Wilfried Schmid \cite{miller2008unpublished}. I thank Stephen D. Miller for suggesting the problem and for his valuable advice. I also thank Valentin Blomer for helpful suggestions. I thank the referee for careful reading and helpful suggestions. This work was partially supported by ERC Advanced Grant 101054336 and Germany’s 
Excellence Strategy grant EXC-2047/1 - 390685813.
\section{Formula for the birational map}\label{sec:formula}
\subsection{Notation}\label{sec:notations}
For a Weyl group element $w\in \Omega$ in $G=\GL(r)$, we let $wn$ be the general matrix of $wN^{w\i}$, where $n\in N^{w\i}$ is the unipotent upper triangular matrix with free variables $n_{a,b} \in V_w$, as defined in (\ref{nws}) and (\ref{coords}). For example, we have
\begin{equation} \label{ex1:coords}
w_1=\left(
\begin{smallarray}{cccc}
 0 & 0 & 0 & 1_4 \\
 0 & 1_2 & 0 & 0 \\
 0 & 0 & 1_3 & 0 \\
 1_1 & 0 & 0 & 0 \\
\end{smallarray}
\right) \quad \text{and} \quad w_1 n=\left(
\begin{smallarray}{cccc}
 0 & 0 & 0 & 1_4 \\
 0 & 1_2 & 0 & n_{2,4} \\
 0 & 0 & 1_3 & n_{3,4} \\
 1_1 & n_{1,2} & n_{1,3} & n_{1,4} \\
\end{smallarray}
\right) \leadsto \vcenter{\hbox{\begin{tikzpicture}[>=stealth, <-, every node/.style={ minimum size=5mm, inner sep=0}, scale=0.7, transform shape]
  % Define nodes in grid positions
  \node (one4) at (3,3) {$1_4$};
  \node (n24) at (3,2.25) {$n_{2,4}$};
  \node (n34) at (3,1.5) {$n_{3,4}$};
  \node (n14) at (3,0.75) {$n_{1,4}$};
  \node (one2) at (1,2.25) {$1_2$};
  \node (one3) at (2,1.5) {$1_3$};
  \node (one1) at (0,0.75) {$1_1$};
  \node (n12) at (1,0.75) {$n_{1,2}$};
  \node (n13) at (2,0.75) {$n_{1,3}$};

  % Define edges based on the new criteria
  % Row connections (leftward)
  \path[<-] (one1) edge (n12);  % 1_4 to n_1,4
  \path[<-] (n12) edge (n13);  % 1_2 to n_2,4
  \path[<-] (n13) edge (n14);  % 1_3 to n_3,4
  \path[<-] (one3) edge (n34);  % 1_1 to n_1,2
  \path[<-] (one2) edge (n24);   % n_1,2 to n_1,3

  % Column connections (upward)
  \path[<-] (one2) edge (n12);  % n_1,4 to n_2,4
  \path[<-] (one3) edge (n13);  % n_2,4 to n_3,4
\path[<-] (n34) edge (n14); 
\path[<-] (n24) edge (n34); 
\path[<-] (one4) edge (n24); 

\end{tikzpicture}}}.\end{equation}
To distinguish the $1$-entries in different positions, we assign indices to them and write $1_1,1_2,\ldots 1_r$ from the left to the right. Note that the free variables appear in the entries that are below the $1$-entry in its column, and to the right of the $1$-entry in its row. We have $V_{w_1}=\{n_{1,2},\, n_{1,3},\, n_{1,4},\, n_{2,4},\, n_{3,4}\}$. Recall from (\ref{wnentries}) that $n_{a,b}$ is in the $\pi\i(a)$-th row and the $b$-th column of $wn$. Each $1_j$ satisfies $1_j=n_{j,j}$, so $1_j$ is in the $\pi\i(j)$-th row and the $j$-th column. \par
We define rectilinear paths in the matrix $wn$ as follows: we say that two nonzero entries of $wn$ are \textit{neighbors} if they are in the same row or the same column, and if there is no other nonzero entry between them. This way we can regard $wn$ as a directed graph, as illustrated above. For two nonzero entries $n_\a$ and $n_\b$ in $wn$, possibly $1$, a \textit{path} from $n_{\alpha}$ to $n_{\beta}$ is a sequence of nonzero neighboring entries from $n_\alpha$ to $n_{\beta}$ which moves only in upward and leftward steps as it passes through the nonzero entries of $wn$. We define the length of a path $p$ as the number of entries in $p$ minus $1$. \par 
We say that two paths $p_1$ and $p_2$ in $wn$ are \textit{disjoint} if $p_1$ and $p_2$ do not share any entries of $wn$, and that $p_1$ and $p_2$ \textit{intersect} if they have one or more common elements. If $p_1$ and $p_2$ both contain $n_\a\in wn$, we say that $p_1$ and $p_2$ \textit{intersect} at $n_\a$. Extending the notion of disjointness to sets of paths, we say that $\mathbf p=\{p_1, p_2, \ldots, p_m\}$ is a set of disjoint paths if $p_i$ and $p_j$ are disjoint whenever $i\neq j$. For two sets of disjoint paths $\mathbf p$ and $\mathbf q$, we say that $\mathbf p$ and $\mathbf q$ are disjoint if any two paths $p\in \mathbf p$ and $q\in \mathbf q$ are disjoint.
\begin{definition} Let $A$ and $B$ be sets of nonzero entries of $wn$ such that $\abs{A}=\abs{B}=m\geq 1$. We let ${\mathcal P}\left(A\to B\right)$ denote the collection of all possible sets of $m$ disjoint paths connecting elements of $A$ to elements of $B$.
\end{definition}
This means that each of the $m$ elements of $A$ is connected by a path to a unique element of $B$, in a way that no two paths have an entry of $wn$ in common. In the example $w_1n$ above, we have ${\mathcal P}\bigl(\{n_{1,4},n_{3,4}\}\to \{1_2,1_3\}\bigr)$ consists of two sets of disjoint paths,
\[\bigl\{\{(n_{1,4}, n_{1,3}, n_{1,2}, 1_2), (n_{3,4}, 1_3)\}, \{(n_{1,4}, n_{1,3}, 1_{3}), (n_{3,4}, n_{2,4},1_2)\}\bigr\}.\]
If there is an entry of $wn$ that is both in $A$ and $B$, then the entry itself forms a path of length $0$. An example of such a case can be found in \eqref{ex:w2}.
\begin{definition}\label{def:upath}
For any set of disjoint paths $\mathbf{p}$, we let $u(\mathbf{p})$ denote the product of all free variables that are traversed by any of the $m$ paths.
\end{definition} 
For the set $\mathbf{p}=\{(n_{1,4}, n_{1,3}, n_{1,2},1_2), (n_{3,4}, 1_3)\}$ in \eqref{ex1:coords}, we have $u(\mathbf{p})=n_{1,2}n_{1,3}n_{1,4}n_{3,4}$.
\begin{definition}\label{pathsumdef}
We let
$P(A\to B)=\sum\limits_{\mathbf{p}\in{\mathcal P}\left(A\to B\right)}u(\mathbf{p})$.
\end{definition}
\begin{definition}\label{def:gamma1}
For each $1_j$ in $wn$, we let $\gamma(1_j)$ denote the rightmost nonzero entry of the row containing $1_j$. For a set $I$ consisting of the entries $1_j$, we let $\gamma(I)=\{\gamma(1_i) \mid 1_i\in I\}$.
\end{definition}
\begin{definition}\label{rhodef}
For each $1_j$ in $wn$, we let $\rho(1_j)$ denote the product of nonzero entries in the row which contains $1_j$.
\end{definition}
In the example of $w_1n$ in (\ref{ex1:coords}), we have
$\gamma(1_1)=n_{1,4}$ and $\rho(1_1)=n_{1,2}n_{1,3}n_{1,4}$.
\subsection{Formula for birational map}\label{sec:Rformula}
For each $n_{a,b}\in V_w$, we assign ``origins" and ``destinations" corresponding to the variable. Recall that $1_{b}$ is the $1$-entry above $n_{a,b}$. We call $1_{b}$ the ``top destination" of $n_{a,b}$.
\begin{definition}\label{mwna}
The matrix $M_w(n_{a,b})$ denotes the submatrix of $wn$ whose rows consist of the $\pi\i(b)$-th row through the $\pi\i(a)$-th row, and whose columns consist of the $b$-th column through the $r$-th column.
\end{definition}
Observe that the top left corner entry of $M_w(n_{a,b})$ is $1_b$ and the bottom left corner is $n_{a,b}$.
\begin{definition}\label{Ddef}
The set of ``destinations" of $n_{a,b}$, which we denote by $D(n_{a,b})$, is the set of $1$-entries in the submatrix $M_w(n_{a,b})$.
\end{definition} 
\begin{definition}\label{Odef}
The set of ``origins" of $n_{a,b}$, which we denote by $O(n_{a,b})$, is the set consisting of the rightmost nonzero entry of the bottom row of $M_w(n_{a,b})$, and the rightmost nonzero entries of the rows of $M_w(n_{a,b})$ that contain a destination of $n_{a,b}$ other than the top destination $1_b$.
\end{definition}
We call the origin in the bottom row of $M_w(n_{a,b})$ the ``bottom origin". For each row that contains a destination of $n_{a,b}$ other than the top destination, there is also an origin of $n_{a,b}$ which is not the bottom origin. Thus $D(n_{a,b})$ and $O(n_{a,b})$ are nonempty and have the same size, hence the collection ${\mathcal P}\bigl(O(n_\alpha)\to D(n_\alpha)\bigr)$ and the polynomial $P\bigl(O(n_\alpha)\to D(n_\alpha)\bigr)$ are well-defined. 
\begin{definition}\label{def:calPn}
We let ${\mathcal P}(n_\alpha)={\mathcal P}\bigl(O(n_\alpha)\to D(n_\alpha)\bigr)$
and $P(n_\alpha)=P\bigl(O(n_\alpha)\to D(n_\alpha)\bigr)$.
\end{definition}
The formula for the birational map $R$ is given as follows.
\begin{definition}\label{def:Rformula}
    For $n_\a\in V_w$, we let
    \begin{equation}\label{defR}
R(n_\a)=(-1)^{t_\a}\frac{P(n_\a)}{\prod\limits_{1_\mu\in D(n_\alpha)}\rho(1_\mu)},
\quad \text{where}\quad t_\a=\abs{O(n_\alpha)}-1=\abs{D(n_\alpha)}-1.\end{equation}
\end{definition}
We shall use the notations $R(n_\a)$ and $u_\a$ interchangeably. We let $wu$ denote the matrix obtained from $wn$ by replacing $n_\a\in V_w$ with $u_\a$, and denote 
$u_{a,b}=(wu)_{\pi\i(a),b}$. Thus $u_{a,b}=n_{a,b}$ if $n_{a,b}$ is either $0$ or $1$, and $u_{a,b}=R(n_{a,b})$ if $n_{a,b}\in V_w$. \par
For example, consider
\begin{equation}\label{ex:w2}
w_2n=\left(
\begin{smallarray}{cccccc}
\cline{4-6}
 0 & 0 & 0 & \multicolumn{1}{|c}{1_4} & 0 & \multicolumn{1}{c|}{0 }\\
 0 & 0 & 0 & \multicolumn{1}{|c}{0} & 0 & \multicolumn{1}{c|}{1_6} \\
 0 & 1_2 & 0 & \multicolumn{1}{|c}{n_{2,4}} & 0 & \multicolumn{1}{c|}{n_{2,6} }\\
 0 & 0 & 0 & \multicolumn{1}{|c}{0} & 1_5 & \multicolumn{1}{c|}{n_{5,6}} \\
 1_1 & n_{1,2} & 0 & \multicolumn{1}{|c}{n_{1,4}} & n_{1,5} & \multicolumn{1}{c|}{n_{1,6}} \\
 0 & 0 & 1_3 & \multicolumn{1}{|c}{n_{3,4}} & n_{3,5} & \multicolumn{1}{c|}{n_{3,6}} \\
 \cline{4-6}
\end{smallarray}
\right)
\end{equation}
This is the general form of a matrix in $w_2N^{w_2\i}$, where $w_2$ is the permutation matrix obtained by replacing the entries $n_\a$ by $0$. Consider the variable $n_{3,4}$. The box represents the submatrix $M_{w_2}(n_{3,4})$. We have $O(n_{3,4})=\{n_{3,6},n_{5,6},1_6\}$ and $D(n_{3,4})=\{1_4,1_5,1_6\}$. Note that a $1$-entry in $wn$ that does not have any free variables in its row may be regarded as both an origin and a destination of a free variable, as the entry $1_6$ in this example. In such a case, the $1$-entry itself forms a path of length $0$. The change of variable for $n_{3,4}$ consists of $5$ sets of three disjoint paths, represented by
\begin{equation}\label{w2n34}\mathbf{p}_1=\begin{smallarray}{ccc}
\cline{1-3}
\multicolumn{1}{|c|}{ \color{orange}1_4} & 0 & \multicolumn{1}{c|}{0}   \\
\cline{3-3}
 \multicolumn{1}{|c|}{\color{orange}0} & 0 &\multicolumn{1}{|c|} {\color{magenta}1_6} \\
 \cline{3-3}
\multicolumn{1}{|c|} {\color{orange}n_{2,4}} & 0 & \multicolumn{1}{c|}{n_{2,6}}   \\
\cline{2-3}
\multicolumn{1}{|c|}{\color{orange}0} & \multicolumn{1}{c}{\color{blue}1_5} & \multicolumn{1}{c|}{\color{blue}n_{5,6}}    \\
\cline{2-3}
\multicolumn{1}{|c|} {\color{orange}n_{1,4}} & n_{1,5} &\multicolumn{1}{c|} {n_{1,6}}   \\
\cline{2-3}
\multicolumn{1}{|c}  {\color{orange}n_{3,4}} & \color{orange}n_{3,5} & \multicolumn{1}{c|}{\color{orange}n_{3,6}}   \\
\cline{1-3}
\end{smallarray}, \, \mathbf{p}_2=\begin{smallarray}{ccc}
\cline{1-3}
\multicolumn{1}{|c|}{\color{orange} 1_4} & 0 & \multicolumn{1}{c|}{0}   \\
\cline{3-3}
 \multicolumn{1}{|c|}{\color{orange}0} & 0 &\multicolumn{1}{|c|} {\color{magenta}1_6} \\
 \cline{3-3}
\multicolumn{1}{|c|} {\color{orange}n_{2,4}} & 0 & \multicolumn{1}{c|}{n_{2,6}}   \\
\cline{2-3}
\multicolumn{1}{|c|}{\color{orange}0} & \multicolumn{1}{c}{\color{blue}1_5} & \multicolumn{1}{c|}{\color{blue}n_{5,6}}    \\
\cline{2-3}
\multicolumn{1}{|c} {\color{orange}n_{1,4}} & \multicolumn{1}{c|}{\color{orange}n_{1,5}} &\multicolumn{1}{c|} {n_{1,6}}   \\
\cline{1-1} \cline{3-3}
\multicolumn{1}{|c}  {n_{3,4}} & \multicolumn{1}{|c}{\color{orange}n_{3,5}} & \multicolumn{1}{c|}{\color{orange}n_{3,6}}   \\
\cline{1-3}
\end{smallarray}, \, \mathbf{p}_3=\begin{smallarray}{ccc}
\cline{1-3}
\multicolumn{1}{|c|}{ \color{orange}1_4} & 0 & \multicolumn{1}{c|}{0}   \\
\cline{3-3}
 \multicolumn{1}{|c|}{\color{orange}0} & 0 &\multicolumn{1}{|c|} {\color{magenta}1_6} \\
 \cline{3-3}
\multicolumn{1}{|c|} {\color{orange}n_{2,4}} & 0 & \multicolumn{1}{c|}{n_{2,6}}   \\
\cline{2-3}
\multicolumn{1}{|c|}{\color{orange}0} & \multicolumn{1}{c}{\color{blue}1_5} & \multicolumn{1}{c|}{\color{blue}n_{5,6}}    \\
\cline{2-3}
\multicolumn{1}{|c} {\color{orange}n_{1,4}} & \multicolumn{1}{c}{\color{orange}n_{1,5}} &\multicolumn{1}{c|} {\color{orange}n_{1,6}}   \\
\cline{1-2}
\multicolumn{1}{|c}  {n_{3,4}} & \multicolumn{1}{c|}{n_{3,5}} & \multicolumn{1}{c|}{\color{orange}n_{3,6}}   \\
\cline{1-3}
\end{smallarray},\, \mathbf{p}_4=\begin{smallarray}{ccc}
\cline{1-3}
\multicolumn{1}{|c|}{ \color{blue}1_4} & 0 & \multicolumn{1}{c|}{0}   \\
\cline{3-3}
 \multicolumn{1}{|c|}{\color{blue}0} & 0 &\multicolumn{1}{|c|} {\color{magenta}1_6} \\
 \cline{2-3}
\multicolumn{1}{|c} {\color{blue}n_{2,4}} & \color{blue}0 & \multicolumn{1}{c|}{\color{blue}n_{2,6}}   \\
\cline{1-2}
\multicolumn{1}{|c|}{0} & \color{orange}1_5 & \multicolumn{1}{|c|}{\color{blue}n_{5,6}}    \\
\cline{3-3}
\multicolumn{1}{|c|} {n_{1,4}} & \color{orange}n_{1,5} &\multicolumn{1}{c|} {\color{orange}n_{1,6}}   \\
\cline{2-2}
\multicolumn{1}{|c}  {n_{3,4}} & n_{3,5} & \multicolumn{1}{|c|}{\color{orange}n_{3,6}}   \\
\cline{1-3}
\end{smallarray}, \,  \mathbf{p}_5=\begin{smallarray}{ccc}
\cline{1-3}
\multicolumn{1}{|c|}{ \color{blue}1_4} & 0 & \multicolumn{1}{c|}{0}   \\
\cline{3-3}
 \multicolumn{1}{|c|}{\color{blue}0} & 0 &\multicolumn{1}{|c|} {\color{magenta}1_6} \\
 \cline{2-3}
\multicolumn{1}{|c} {\color{blue}n_{2,4}} &\color{blue} 0 & \multicolumn{1}{c|}{\color{blue}n_{2,6}}   \\
\cline{1-2}
\multicolumn{1}{|c|}{0} & \color{orange}1_5 & \multicolumn{1}{|c|}{\color{blue}n_{5,6}}    \\
\cline{3-3}
\multicolumn{1}{|c|} {n_{1,4}} & \color{orange}n_{1,5} &\multicolumn{1}{|c|} {n_{1,6}}   \\
\cline{3-3}
\multicolumn{1}{|c}  {n_{3,4}} & \multicolumn{1}{|c}{\color{orange}n_{3,5}} & \multicolumn{1}{c|}{\color{orange}n_{3,6}}   \\
\cline{1-3}
\end{smallarray}.\end{equation}
It is straightforward to check that the change of coordinate \eqref{defR} for $n_{3,4}$ gives
\[\begin{aligned}
    u_{3,4}&=n_{1,4} n_{2,4} n_{3,4} n_{3,5} n_{3,6}+n_{1,4} n_{1,5} n_{2,4} n_{3,5} n_{3,6}+n_{1,4} n_{1,5} n_{1,6} n_{2,4} n_{3,6}+
    n_{1,5} n_{1,6} n_{2,4} n_{2,6} n_{3,6} \\ &\quad +n_{1,5} n_{2,4} n_{2,6} n_{3,5} n_{3,6}.
\end{aligned}
\]

\section{More on paths and matrices}\label{sec:notations1}
Throughout the section, we fix a Weyl group element $w\in \Omega$ of $\GL(r)$. In this section, we describe a partition of $\mathcal{P}(n_\a)$ into three subcollections and explain its use in later analysis.
\subsection{Subsets of $O(n_\a)$ and $D(n_\a)$}
Recall Definitions~\ref{Ddef} and~\ref{Odef} of the sets of destinations and origins of $n_\a$. We first establish notations for subsets of $O(n_\a)$ and $D(n_\a)$.
\begin{definition}\label{O1}
We let $O_1(n_\alpha)\subseteq O(n_\a)$ denote the subset consisting of the origins which are (i) in the rightmost column of $wn$, (ii) above the bottom origin, and (iii) right below a free variable which is not an origin of $n_\a$. That is,
\[O_1(n_{a,b})=\{n_{\pi(t),r}\in O(n_{a,b})\mid t<\pi\i(a),\  n_{\pi(t-1),r}\in V_w\smallsetminus O(n_{a,b})\}.\]
\end{definition}
Observe that every element of $O_1(n_{a,b})$ is a free variable because they are right below a free variable.
\begin{definition}\label{D1}
We let $D_1(n_\a)\subseteq D(n_\a)$ denote the subset consisting of the $1$-entries in the rows of $wn$ containing an element of $O_1(n_\a)$.
\end{definition}
\begin{definition}\label{ODup}
For a positive integer $k\leq r$, we let $O_{\uparrow,k}(n_\a)$ and $D_{\uparrow,k}(n_\a)$ denote the set of origins and destinations of $n_\a$ which are in or above the $k$-th row of $wn$, respectively. Similarly, we let $O_{\downarrow,k}(n_\a)$ and $D_{\downarrow,k}(n_\a)$ denote the set of origins and destinations of $n_\a$ which are in or below the $k$-th row of $wn$, respectively.
\end{definition}
\begin{definition}\label{Dne}
For a positive integer $k\leq r$, let $D_{\nearrow,k}(n_\a)$ be the set of the destinations of $n_\a$ which are either equal to $1_k$, or above and to the right of $1_k$.
\end{definition}
See section 4.2 of \cite{kimWhit} for an example illustrating these subsets.

\subsection{Partition of $\mathcal{P}(n_\a)$} We partition the collection $\mathcal{P}(n_\a)$ of sets of disjoint paths according to the initial direction of the paths in each set.

\begin{definition}\label{calpleft}
We let ${\mathcal P}_L\left(A\to B\right)\subseteq {\mathcal P}\left(A\to B\right)$ denote the subcollection consisting of the sets in which every path with length at least $1$ whose origin is in the rightmost column of $wn$ starts by going left. We let
$P_L(A\to B)=\sum_{\mathbf{p}\in{\mathcal P}_L\left(A\to B\right)}u(\mathbf{p})$.
\end{definition}
\begin{definition}\label{def:calpleftn}
We let ${\mathcal P}_L(n_\a)={\mathcal P}_L\bigl(O(n_\a)\to D(n_\a)\bigr)$ and $P_L(n_\a)=P_L\bigl(O(n_\a)\to D(n_\a)\bigr)$.
\end{definition}
Since the entries in the rightmost column of $wn$ are zero if they are above $1_r$ and free variables if they are below $1_r$, the elements of $O(n_\a)$ are in the rightmost column if and only if they are below the $\pi\i(r)$-th row. If all of the origins of $n_\a$ are above the $\pi\i(r)$-th row, then ${\mathcal P}(n_\a)={\mathcal P}_L(n_\a)$. Write ${\mathcal P}'(n_\a)={\mathcal P}(n_\a)\smallsetminus {\mathcal P}_L(n_\a)$. That is, ${\mathcal P}'(n_\a)$ is the subcollection of ${\mathcal P}(n_\a)$ consisting of the sets that contain a path of length at least $1$ that starts by going up, from an origin in the rightmost column of $wn$.
\begin{definition}\label{P1def}
Let ${\mathcal P}_1(n_\a)$ denote the subcollection of ${\mathcal P}'(n_\a)$ consisting of sets in which the path from the \textbf{bottom} origin starts by going left. Let $P_1(n_\a)=\sum_{\mathbf{p}\in{\mathcal P}_1(n_\a)}u(\mathbf{p})$.
\end{definition}
\begin{definition}\label{P2def}
Let ${\mathcal P}_2(n_\a)$ denote the subcollection of ${\mathcal P}'(n_\a)$ consisting of sets in which the path from the \textbf{bottom} origin starts by going up. Let $P_2(n_\a)=\sum_{\mathbf{p}\in{\mathcal P}_2(n_\a)}u(\mathbf{p})$.
\end{definition}
It is clear that ${\mathcal P}'(n_\a)={\mathcal P}_1(n_\a)\sqcup{\mathcal P}_2(n_\a)$. If a set $\mathbf{p}$ is in ${\mathcal P}_1(n_\a)$, then by Definition~\ref{P1def}, the path in $\mathbf{p}$ starting from the bottom origin starts by going left, but there is another path $p\in \mathbf{p}$ that starts by going up, whose origin is in the rightmost column and above the bottom origin. Hence $\mathbf{p}$ necessarily has at least two paths of length at least $1$. Also, since $p$ starts by going up, the second entry of $p$ is right above its origin, and it is not an origin of $n_\a$ due to the disjointness. Hence the origin of $p$ is right below a free variable which is not an origin of $n_\a$, therefore it is in the set $O_1(n_\a)$. \par
We have $P(n_\a)=P_L(n_\a)+P_1(n_\a)+P_2(n_\a)$. If ${\mathcal P}'(n_\a)$ is empty, then $P(n_\a)=P_L(n_\a)$. In the example \eqref{w2n34}, we have $\mathbf{p}_1,\mathbf{p}_2\in \mathcal{P}_L(n_{3,4})$, $\mathbf{p}_5\in \mathcal{P}_1(n_{3,4})$, and $\mathbf{p}_3,\mathbf{p}_4\in \mathcal{P}_2(n_{3,4})$.
\begin{definition}\label{def:RLR1R2}
For a subscript $\bullet \in \{L,1,2\}$, we let
\[
  R_{\bullet}(n_\a)=(-1)^{t_\a} \frac{P_{\bullet}(n_\a)}{\prod\limits_{1_\mu\in D(n_\alpha)}\rho(1_\mu)}, \quad t_\a=\abs{D(n_\alpha)}-1.
\]
\end{definition}
Using this notation, we can write $R(n_\a)$ defined in (\ref{defR}) as the sum
\begin{equation}\label{Rpartition}
R(n_\a)=R_L(n_\a)+R_1(n_\a)+R_2(n_\a). 
\end{equation}
Among the three terms, $R_L(n_\a)$ admits a simple formula in terms of a free variable of a Weyl group element $\wtilde$ of $\GL(r-1)$. We explain this in the next section.
\subsection{From $w$ to $\wtilde$} Recall that $1_r$, the rightmost $1$-entry in $w$, is in the $\pi\i(r)$-th row. Let $\widetilde{w}$ be the Weyl group element of size $(r-1)$ obtained from $w$ by removing the rightmost column and the $\pi\i(r)$-th row. To make a distinction between $w$ and $\wtilde$, we use superscripts $(w)$ and $(\wtilde)$ to previously defined notions. For example, we write $R(n_\a)=R^{(w)}(n_\a)$ if $n_\a$ is an element of $V_w$.
\begin{definition}\label{def:wtildenu}
We let $\wtilde\ntilde$ be the general matrix of $\wtilde \bar{N}^{\wtilde\i}$, where $\bar{N}\subset \GL(r-1)$ is the group of unipotent upper triangular matrices of size $r-1$, and $\ntilde\in \bar{N}^{\wtilde\i}$ is the unipotent upper triangular matrix with free variables $n_{\atilde,\btilde} \in V_{\wtilde}$, as defined in (\ref{nws}) and (\ref{coords}). We let $\wtilde\utilde$ be the matrix obtained from $\wtilde \ntilde$ by replacing the free variables $n_{\atilde,\btilde}$ in $\wtilde \ntilde$ with $u_{\atilde,\btilde}=R^{(\wtilde)}(n_{\atilde,\btilde})$.
\end{definition}
Consider again the example \eqref{ex:w2}. Removing the second row and the rightmost column from $w_2n$ gives
\[
\left(
\begin{smallarray}{ccccc}
 0 & 0 & 0 & 1_4 & 0 \\
 0 & 1_2 & 0 & n_{2,4} & 0 \\
 0 & 0 & 0 & 0 & 1_5 \\
 1_1 & n_{1,2} & 0 & n_{1,4} & n_{1,5} \\
 0 & 0 & 1_3 & n_{3,4} & n_{3,5} \\
\end{smallarray}
\right).
\]
This is precisely the general matrix $\widetilde{w}_2\ntilde$, hence the set $V_{\widetilde{w}_2}$ consists of the elements of $V_{w_2}$ that are not in the rightmost column of $w_2 n$. Consider $n_{3,4}$ as we did before in \eqref{ex:w2}, but this time as a free variable in $\wtilde_2\ntilde$. We have $D^{(\wtilde_2)}(n_{3,4})=\{1_4,1_5\}$ and $O^{(\wtilde_2)}(n_{3,4})=\{n_{3,5},1_5\}$. Then the collection $\mathcal{P}^{(\wtilde_2)}(n_{3,4})$ consists of two sets of disjoint paths, represented by
\[\widetilde{\mathbf{p}}_1=\begin{smallarray}{cc}
\cline{1-2}
\multicolumn{1}{|c|}{ \color{orange}1_4} & \multicolumn{1}{c|}{0}   \\
\multicolumn{1}{|c|} {\color{orange}n_{2,4}} & \multicolumn{1}{c|}{0}    \\
\cline{2-2}
\multicolumn{1}{|c|}{\color{orange}0} & \multicolumn{1}{c|}{\color{blue}1_5}     \\
\cline{2-2}
\multicolumn{1}{|c|} {\color{orange}n_{1,4}} & \multicolumn{1}{c|}{n_{1,5}}  \\
\cline{2-2}
\multicolumn{1}{|c}  {\color{orange}n_{3,4}} & \multicolumn{1}{c|}{\color{orange}n_{3,5} }  \\
\cline{1-2}
\end{smallarray} \eqand \widetilde{\mathbf{p}}_2=\begin{smallarray}{cc}
\cline{1-2}
\multicolumn{1}{|c|}{\color{orange} 1_4} & \multicolumn{1}{c|}{0}   \\
\multicolumn{1}{|c|} {\color{orange}n_{2,4}} & \multicolumn{1}{c|}{0}   \\
\cline{2-2}
\multicolumn{1}{|c|}{\color{orange}0} & \multicolumn{1}{c|}{\color{blue}1_5}    \\
\cline{2-2}
\multicolumn{1}{|c} {\color{orange}n_{1,4}} & \multicolumn{1}{c|}{\color{orange}n_{1,5}}    \\
\cline{1-1}
\multicolumn{1}{|c}  {n_{3,4}} & \multicolumn{1}{|c|}{\color{orange}n_{3,5}}  \\
\cline{1-2}
\end{smallarray}.\]
Observe that removing the origins from the paths in $\mathbf{p}_1$ and $\mathbf{p}_2$ in \eqref{w2n34} gives $\widetilde{\mathbf{p}}_1$ and $\widetilde{\mathbf{p}}_2$, respectively. Recall from the remark below Definition~\ref{P2def} that, as an entry of $w_2n$, we have $\mathcal{P}^{(w_2)}_L(n_{3,4})=\{\mathbf{p}_1,\mathbf{p}_2\}$. It follows that $P^{(w_2)}_L(n_{3,4})=n_{3,6}n_{5,6}P^{(\wtilde_2)}(n_{3,4})$. We record the results in the following lemma.
\begin{lem}\label{lem:PLwPwtilde}
The set of free variables in $V_w$ which are not in the rightmost column is in bijection with the set $V_{\wtilde}$, that is, $V_{\wtilde}=\{n_{a,b}\in V_w\mid b<r\}$. If $n_{a,b}\in V_w$ is not in the rightmost column of $wn$, then the sets in ${\mathcal P}^{(w)}_L(n_{a,b})$ are in one-to-one correspondence with the sets in ${\mathcal P}^{(\wtilde)}(n_{a,b})$. Also, we have 
\begin{equation}\label{PLwPwtilde}
    P_L^{(w)}(n_{a,b})=\Bigl(\prod\limits_{n_{\b,r}\in O^{(w)}_{\downarrow, \pi^{-1}(r)+1}(n_{a,b})} n_{\b,r}\Bigr)\cdot P^{(\widetilde{w})}(n_{a,b})
\end{equation}
for all such $n_{a,b}\in V_w$.
\end{lem}
\begin{proof} The statement on $V_{\wtilde}$ is straightforward. For \eqref{PLwPwtilde}, observe that there is an obvious one-to-one correspondence between $\mathcal{P}^{(\wtilde)}(n_\a)$ and $\mathcal{P}_L^{(w)}(n_\a)$. By construction, every set of disjoint paths in $\mathcal{P}_L^{(w)}(n_\a)$ contains the origins of $n_\a$ in the rightmost column but no other elements in the rightmost column. Removing these origins gives a set of disjoint paths in $\mathcal{P}^{(\wtilde)}(n_\a)$. Conversely, by attaching the elements of $O^{(w)}(n_\a)$ in the rightmost column to the paths in a set $\mathcal{P}^{(\wtilde)}(n_\a)$, and adding the path $(1_r)$ of length $0$ if necessary, we obtain a set of disjoint paths in $\mathcal{P}_L^{(w)}(n_\a)$. The equation \eqref{PLwPwtilde} is a direct consequence of this one-to-one correspondence.
\end{proof}
\noindent See (6.143) through (6.145) in \cite{kimWhit} for more details. From this lemma, we can derive the following formula for $R_L(n_{a,b})$.
\begin{cor}[Lemma 6.141 of \cite{kimWhit}]\label{RLwRwtilde}
Let $n_{a,b}\in V_w$ be a free variable that is not in the rightmost column of $wn$. The function $R_L=R^{(w)}_L$ satisfies
\begin{equation}\label{RLwRwtilde0}
R^{(w)}_L(n_{a,b})=\begin{cases}
R^{(\widetilde{w})}(n_{a,b}) & \text{if} \quad \pi\i(b)<\pi\i(a)<\pi\i(r), \\
-n_{a,r}\cdot R^{(\widetilde{w})}(n_{a,b}) & \text{if} \quad \pi\i(b)<\pi\i(r)<\pi\i(a), \\
\frac{n_{a,r}}{n_{b,r}}\cdot R^{(\widetilde{w})}(n_{a,b}) & \text{if} \quad \pi\i(r)<\pi\i(b)<\pi\i(a).
\end{cases}
\end{equation}
\end{cor}
\noindent For brevity, we postpone the proof to Appendix~\ref{sec:technical}. The corollary tells us that among the three terms in \eqref{Rpartition}, $R_L(n_\a)$ is the one that contains the information of the rational map $R$ for a Weyl group element of one size smaller. Thus by getting rid of the terms $R_1(n_\a)$ and $R_2(n_\a)$ we can apply induction to compute the determinants of the submatrices of $wu$ appearing in \eqref{superdiag}. We achieve this by a series of row and column operations.
\subsection{The row and column operations}
Let $1\leq i\leq r-1$. We begin by defining the submatrices of $wn$ and $wu$ that are relevant to the formulas in \eqref{superdiag}.
\begin{definition}\label{defSi}
We let $S_i$ denote the lower right square submatrix of $wn$ of size $i$ whose rows consist of the $(r-i+1)$-th through the $r$-th rows of $wn$ and columns consist of the $(r-i+1)$-th through the $r$-th columns of $wn$.
\end{definition}
\begin{definition}\label{defEi}
We let $E_i$ denote the square submatrix of $wn$ of size $i$, whose rows consist of the $(r-i)$-th row and the $(r-i+2)$-th through the $r$-th rows of $wn$, and columns consist of the $(r-i+1)$-th through the $r$-th columns of $wn$.
\end{definition}
\begin{definition}\label{defTi}
We let $T_i$ and $F_i$ be the submatrices of $wu=R(wn)$ whose entries correspond to those of $S_i$ and $E_i$, respectively. 
\end{definition}
The formulas in \eqref{superdiag} can be written as $h_{i,i}=\frac{\det(T_{r-i+1})}{\det(T_{r-i})}$ and $x_{i,i+1}=\frac{\det (F_{r-i})}{\det (T_{r-i})}$.
\begin{definition}\label{xkdef}
For $k\geq r-i$, we let $x_k$ denote the row vector of length $i$ which consists of the last $i$ entries of the $k$-th row of $wu$, that is, $x_k=(wu)_{k,r-i+1\leq j\leq r}$. For $r-i+1\leq j\leq r$, write $x_k(j)=(wu)_{k,j}=u_{\pi(k),j}$.
\end{definition}
Check that $T_i$ consist of the rows $x_{r-i+1}$ through $x_r$, and $F_i$ consists of the rows $x_{r-i}$ and $x_{r-i+2}$ through $x_r$.
\begin{definition}\label{x'kdef}
    For $k\geq r-i$, $k\neq \pi\i(r)$, we let $x'_k$ denote the row matrix of length $i-1$, \[x'_k=[x'_k(r-i+1) \quad x'_k(r-i+2) \quad \cdots \quad x'_k(r-1)],\] where
    \begin{equation}\label{x'kdef1}
    x'_k(b)=\begin{cases}
    R_L(n_{\pi(k),b})+R_1(n_{\pi(k),b}) & \text{if} \quad x_k(b)=u_{\pi(k),b} \text{ with }\left(\pi(k),b\right)\in Inv(\pi\i) \\
    -n_{\pi(k),r} \cdot x_{k-1}(b) 
    & \text{if}\quad x_k(b)=0 \text{ and }x_{k-1}(b)\neq 0 \\
    x_k(b) & \text{otherwise}
    \end{cases} 
    \end{equation}
    for $r-i+1\leq b\leq r-1$.
\end{definition}
The entry $x_k(b)$ belongs to the first case if $n_{\pi(k),b}$ is a free variable. The entry $x_k(b)$ belongs to the second case if $u_{\pi(k),b}$ is a zero entry which is right below a nonzero entry. The entry $x_k(b)$ belongs to neither of these cases if $x_k(b)=1_b$, or $x_k(b)$ is a zero entry right below a zero entry, that is, if $x_k(b)=x_{k-1}(b)=0$. \par
Let $x_{r-i+s}$ be the row of $T_i$ which contains the highest nonzero entry of the rightmost column. Since the entries in the rightmost column above $1_r$ are all zero, we have $r-i+s=r-i+1$ if $r-i+1>\pi\i(r)$, and $r-i+s=\pi\i(r)$ otherwise.
\begin{definition}\label{Ti'def} We let $T'_i$ be the matrix of size $(i-1)\times (i-1)$ which equals
\[
\left(\begin{smallarray}{c}
 x'_{r-i+2} \\ \vdots \\
 x'_r
\end{smallarray}\right) \quad \text{or} \quad \left(\begin{smallarray}{c}
 x'_{r-i+1} \\ \vdots \\  x'_{r-i+s-1} \\ x'_{r-i+s+1} \\ \vdots \\
 x'_r
\end{smallarray}\right) \quad \text{with} \quad r-i+s=\pi\i(r),\]
depending on whether $r-i+1>\pi\i(r)$ or $r-i+1\leq \pi\i(r)$, respectively.
\end{definition}
\begin{definition}\label{Fi'def1}
We let $F'_i$ be the matrix of size $(i-1)\times (i-1)$ which equals
\[
\left(\begin{smallarray}{c}
 x'_{r-i+1} \\  x'_{r-i+3} \\ \vdots \\ x'_r
\end{smallarray}\right),\quad \left(\begin{smallarray}{c}
 x'_{r-i} \\  x'_{r-i+3} \\ \vdots \\ x'_r
\end{smallarray}\right), \quad\text{or}\quad 
\left(\begin{smallarray}{c}
  x'_{r-i} \\ 
   x'_{r-i+2}  \\ 
   \vdots \\ x'_{r-i+s-1}  \\  x'_{r-i+s+1}  \\
  \vdots\\ 
  x'_r 
\end{smallarray}\right) \quad \text{with} \quad r-i+s=\pi\i(r),
   \]
   depending on whether $r-i+1>\pi\i(r)$, $r-i+1=\pi\i(r)$, or $r-i+1<\pi\i(r)$, respectively.
   \end{definition}
\begin{definition}\label{xLkdef}
    For $k\geq r-i$, $k\neq \pi\i(r)$, we let $x^L_k$ denote the row vector of length $(i-1)$, \[x^L_k=[x^L_k(r-i+1) \quad x^L_k(r-i+2) \quad \cdots \quad x^L_k(r-1)],\] where
    \begin{equation}\label{xLkdef1}
    x^L_k(b)=\begin{cases}
    R_L(n_{\pi(k),b}) & \text{if} \quad \left(\pi(k),b\right)\in Inv(\pi\i) \\
    x_k(b) & \text{otherwise}
    \end{cases} 
    \end{equation}
    for $r-i+1\leq b\leq r-1$.
\end{definition}
\begin{definition}\label{TLFLdef}
We let $T^L_i$ and $F^L_i$ denote the matrices of size $(i-1)\times (i-1)$ obtained from $T'_i$ and $F'_i$, respectively, by replacing the rows $x'_k$ in the matrices $T'_i$ and $F'_i$ with $x^L_k$.
\end{definition}
We prove Proposition~\ref{thm:rowcolop} in two steps. First, we apply a series of row operations to $T_i$ and $F_i$ to write their determinants in terms of $T'_i$ and $F'_i$.
\begin{prop}\label{prop:rowop}
Let $1\leq i\leq r$ satisfy $r-i+1>\pi(r)$, so that the bottom row of $T_i$ and $F_i$ does not contain $1_{\pi(r)}$, the $1$-entry that is in the bottom row of $wu$.
 \begin{enumerate}
        \item[(i)] We have
$\det (T_i)=(-1)^{s+i} \cdot u_{\pi(r-i+s),r}\cdot \det( T_i')$, where $s$ is such that $x_{r-i+s}$ is the row of $T_i$ that contains the highest nonzero entry of the rightmost column of $T_i$.
\item[(ii)] We have
\[\det (F_i)=\begin{cases}
    \frac{1}{n_{\pi(r-i+1),r}}\det(T_i)-\frac{(-1)^{i}}{n_{\pi(r-i+1),r}}\cdot u_{\pi(r-i+2),r}\cdot \det (F'_i) & \text{ if } r-i+1>\pi\i(r) \\
    (-1)^{i}\cdot u_{\pi(r-i+2),r}\cdot \det (F'_i) & \text{ if } r-i+1=\pi\i(r) \\
   (-1)^{r-\pi\i(r)}\cdot \det (F'_i) & \text{ if } r-i+1<\pi\i(r).
\end{cases}\]
    \end{enumerate}
\end{prop}
After the row operations, we perform a series of column operations to reduce $T'_i$ and $F'_i$ to $T^L_i$ and $F^L_i$, respectively.
\begin{prop}\label{prop:colop}
Let $1\leq i\leq r$ satisfy $r-i+1>\pi(r)$. There exists a lower triangular matrix $M_{\text{col}}$ with ones in the diagonal entries, which represents the column operations on $T'_i$, such that
$T'_i\cdot M_{\text{col}}=T^L_i$, and $F'_i\cdot M_{\text{col}}=F^L_i$.
\end{prop}
It is clear that the two propositions imply Proposition~\ref{thm:rowcolop}. The computation of the determinants of $T_i$ and $F_i$ for which $r-i+1\leq\pi(r)$ does not require any row and column operations and will be handled separately in the second paper.
\section{Row operations on $T_i$ and $F_i$}\label{sec:rowop}
In this section, we prove Proposition~\ref{prop:rowop}. Throughout this section, we assume that $r\neq \pi(r)$. This implies $1_{\pi(r)}\neq 1_r$, so the $1$-entry in the bottom row of $wu$ is not in the rightmost column. Consequently there exists an $i$ for which the bottom row of $T_i$ and $F_i$ does not contain the entry $1_{\pi(r)}$. This occurs precisely when $1\leq i<r-\pi(r)+1$. Fix such an $i$. As explained above Definition~\ref{Ti'def}, $T_i$ takes the form of either
\[
T_i=\left(\begin{smallarray}{cccc}
\star& \cdots & \star & u_{\pi(r-i+1),r} \\ 
\star & \cdots & \star & u_{\pi(r-i+2),r}\\ 
\vdots & \cdots & \vdots & \vdots \\
\star & \cdots & \star & u_{\pi(r),r} 
\end{smallarray}\right) \quad \text{or} \quad 
T_i=\left(\begin{smallarray}{cccc}
\star& \cdots & \star & 0 \\ 
\vdots & \cdots & \vdots & \vdots \\
\star& \cdots & \star & 0 \\
0 & \cdots & 0 & 1_r\\
\star & \cdots & \star & u_{\pi(r-i+s+1),r}\\ 
\vdots & \cdots & \vdots & \vdots \\
\star & \cdots & \star & u_{\pi(r),r} 
\end{smallarray}\right),   
\]
depending on whether $r-i+1>\pi\i(r)$ or $r-i+1\leq \pi\i(r)$, respectively. Recall (\ref{Rpartition}). The row operations on $T_i$ and $F_i$ eliminate the terms $R_2(n_\a)$ from the entries $u_\alpha$ with $\alpha\in Inv(\pi\i)$.
\begin{lem}\label{rowop1}
Let $\left(\pi(a),b\right)\in Inv(\pi\i)$. 
If $a<\pi\i(r)$, then $n_{\pi(a),r}=0$ and $R_1(n_{\pi(a),b})=R_2(n_{\pi(a),b})=0$,
hence $n_{\pi(a),b}$ satisfies
\begin{equation}\label{nabove1r2}
  u_{\pi(a),b}-n_{\pi(a),r}\cdot u_{\pi(a-1),b}=u_{\pi(a),b}=R_L(n_{\pi(a),b}).
\end{equation}
If $a>\pi\i(r)$, then
\begin{equation}\label{rowop1eq}
u_{\pi(a),b}-n_{\pi(a),r}\cdot u_{\pi(a-1),b}=R_L(n_{\pi(a),b})+R_1(n_{\pi(a),b}).  
\end{equation}
\end{lem}
\begin{proof}
If $a<\pi\i(r)$, then $n_{\pi(a),b}$ is above the row that contains $1_r$. Since every entry above $1_r$ is zero, $O(n_{\pi(a),b})$ does not contain any variable on the rightmost column. By the remark below Definition~\ref{P2def}, we have $R_1(n_{\pi(a),b})=R_2(n_{\pi(a),b})=0$, and \eqref{nabove1r2} follows immediately. \par
Now, suppose that $a>\pi\i(r)$. Then the bottom origin of $n_{\pi(a),b}$ is $n_{\pi(a),r}\in V_w$. The entry $n_{\pi(a-1),b}$, which is right above $n_{\pi(a),b}$, is either $0$, $1_b$, or a free variable. We present the proof of \eqref{rowop1eq} for the third case and omit the first two, as they follow the same logic but are simpler (see \cite[Lemma 5.4]{kimWhit} for the first two cases). Henceforth assume that assume that $n_{\pi(a-1),b}\in V_w$. In this case, we have $D(n_{\pi(a-1),b})=D(n_{\pi(a),b})$ and 
$O(n_{\pi(a-1),b})=\{n_{\pi(a-1),r}\}\cup O(n_{\pi(a),b})\smallsetminus\{n_{\pi(a),r}\}$. \par 
Let $\mathbf p\in {\mathcal P}_2(n_{\pi(a),b})$. Then by Definition~\ref{P2def} the path $p\in \mathbf p$ from $n_{\pi(a),r}$ starts by going up. We can regard $p$ as a path starting from $n_{\pi(a-1),r}$, with $n_{\pi(a),r}$ prepended to the beginning. Hence the set $\mathbf p'$ obtained by removing $n_{\pi(a),r}$ from the path $p$ is in the collection ${\mathcal P}(n_{\pi(a-1),b})$. Conversely, let $\mathbf q\in {\mathcal P}(n_{\pi(a-1),b})$, and let $q\in \mathbf q$ be the path which starts from $n_{\pi(a-1),r}$. The set $\mathbf q'$ obtained by prepending the path $q$ by $n_{\pi(a),r}$ is in ${\mathcal P}_2(n_{\pi(a),b})$. This shows that ${\mathcal P}_2(n_{\pi(a),b})$ is in one-to-one correspondence with ${\mathcal P}(n_{\pi(a-1),b})$. It follows that
    $P_2(n_{\pi(a),b})=P(n_{\pi(a-1),b})\cdot n_{\pi(a),r}$. Since $n_{\pi(a),b}$ and $n_{\pi(a-1),b}$ have the same set of destinations, it follows from Definition~\ref{def:RLR1R2} that we have $R_2(n_{\pi(a),b})=R(n_{\pi(a-1),b})\cdot n_{\pi(a),r}$, from which \eqref{rowop1eq} follows.
\end{proof}
\begin{cor}[Lemma 5.14 of \cite{kimWhit}]\label{rowopcor}
Recall the Definition~\ref{x'kdef} of $x'_k$. For $k\geq r-i$, $k\neq \pi\i(r)$, we have $x_{k}-n_{\pi(k),r}\cdot x_{k-1}=[x_{k}'\ 0]$. In particular, if $k<\pi\i(r)$, then $n_{\pi(k),r}=0$, and $x_k(b)=x'_k(b)$ for all $r-i+1\leq b\leq r-1$.
\end{cor}
\noindent The proof is straightforward and therefore omitted.
\subsection{Proof of Proposition~\ref{prop:rowop}}
First, to prove part (i), we apply row operations to $T_i$ to put zeros in all but one entry of the rightmost column. Let $x_{r-i+s}$ be the row of $T_i$ that contains the highest nonzero entry of the rightmost column. For each row $x_k$ below $x_{r-i+s}$, we subtract $n_{\pi(k),r}$ times the row $x_{k-1}$ right above $x_k$. The row operations preserve the determinant. By Corollary~\ref{rowopcor}, the row operations change the rows $x_k$ with $k>r-i+s$ to $[x'_k \ 0]$. It also follows from Corollary~\ref{rowopcor} that $x_k=[x'_k \ 0]$ for all $k<r-i+s$. Therefore, after the row operations $T_i$ becomes the form of either
\begin{equation}\label{Ti'form}
    \left(\begin{smallarray}{cc}
\star \, \cdots \, \star & u_{\pi(r-i+1),r} \\ x'_{r-i+2} &  0\\ 
\vdots& \vdots \\
x'_{r} & 0 
\end{smallarray}\right) \quad \text{or} \quad \left(\begin{smallarray}{cc}
 x'_{r-i+1} & 0  \\ 
 \vdots & \vdots \\
 x'_{r-i+s-1} & 0 \\ 
0 \,\, \cdots \,\, 0 & 1_r\\
 x'_{r-i+s+1} & 0 \\ 
 \vdots & \vdots \\
 x'_{r} & 0 
\end{smallarray}\right), 
\end{equation}
depending on whether $s=1$ or $s>1$. Removing the rightmost column and the row $x_{r-i+s}$ from the above matrices gives $T'_i$. This proves part (i). \par
For part (ii), first consider the case $r-i+1>\pi\i(r)$. Then $1_r$ is above the $(r-i+1)$-th row of $wu$, so $n_{\pi(k),r}$ is a free variable for all $k\geq r-i+1$. By Corollary~\ref{rowopcor}, for each $k\geq r-i+1$ we have
$ x_{k}-n_{\pi(k),r}\cdot x_{k-1}=[x_{k}'\ 0]$.
In particular, we have
\[x_{r-i+1}-n_{\pi(r-i+1),r}\cdot x_{r-i}=[x'_{r-i+1} \ 0],\]
which implies
\[
\det(F_i)=\det\left(\begin{smallarray}{c}
 x_{r-i} \\ x_{r-i+2} \\ \vdots \\ x_r
\end{smallarray}\right) =\frac{1}{n_{\pi(r-i+1),r}}\cdot\det\left(\begin{smallarray}{c}
 x_{r-i+1} \\ x_{r-i+2} \\ \vdots \\ x_r
\end{smallarray}\right)-\frac{1}{n_{\pi(r-i+1),r}}\cdot\det\left(\begin{smallarray}{cc}
 x'_{r-i+1} & 0 \\ \multicolumn{2}{c}{x_{r-i+2}} \\ \multicolumn{2}{c}{\vdots} \\ \multicolumn{2}{c}{x_r}
\end{smallarray}\right).
\]
The first matrix on the right is $T_i$. We apply the row operations to the second matrix: we subtract $n_{\pi(r),r}\cdot x_{r-1}$ from $x_r$, $n_{\pi(r-1),r}\cdot x_{r-2}$ from $x_{r-1}$, proceeding up until we apply these row operations to all of the rows which are below $x_{r-i+2}$. By Corollary~\ref{rowopcor}, this reduces $x_k$ with $k>r-i+2$ to $[x'_k \ 0]$. This operation preserves the determinant, hence
\[
\det(F_i)=\frac{1}{n_{\pi(r-i+1),r}}\cdot\det(T_i)-\frac{1}{n_{\pi(r-i+1),r}}\cdot\det\left(\begin{smallarray}{cc}
 x'_{r-i+1} & 0 \\ \multicolumn{2}{c}{x_{r-i+2}} \\ x'_{r-i+3} & 0\\ \multicolumn{2}{c}{\vdots} \\ x'_r & 0
\end{smallarray}\right).
\]
Removing the row $x_{r-i+2}$ and the rightmost column from the second matrix on the right gives $F'_i$. From this, it is clear that its determinant equals $(-1)^{i}\cdot u_{\pi(r-i+2),r}\cdot \det(F'_i)$. This proves the first case of part (ii). \par
The second and the third cases of part (ii) are proved similarly. By Corollary~\ref{rowopcor}, $F_i$ has the form of either
\[
\left(\begin{smallarray}{cc}
 x'_{r-i} & 0 \\ \multicolumn{2}{c}{x_{r-i+2}} \\ \multicolumn{2}{c}{\vdots} \\ \multicolumn{2}{c}{x_r}
\end{smallarray}\right) \quad \text{or}\quad \left(\begin{smallarray}{cc}
   x'_{r-i} & 0 \\ x'_{r-i+2} & 0 \\ \vdots & \vdots \\ x'_{r-i+s-1} & 0 \\ 0\,\, \cdots\,\, 0 & 1_r \\ \multicolumn{2}{c}{x_{r-i+s+1}} \\ \multicolumn{2}{c}{\vdots} \\ \multicolumn{2}{c}{x_r}
\end{smallarray}\right) \quad \text{with}\quad r-i+s=\pi\i(r),
\]
depending on whether $r-i+1=\pi\i(r)$ or $r-i+1<\pi\i(r)$, respectively. In either case, we apply row operations to $F_i$ as follows: we subtract $n_{\pi(r),r}\cdot x_{r-1}$ from $x_r$, $n_{\pi(r-1),r}\cdot x_{r-2}$ from $x_{r-1}$, proceeding up until we apply these row operations to all of the rows below $x_{r-i+2}$ or below $1_r$, respectively. This turns the rows $x_{r-i+j}$ below $x_{r-i+2}$ or $1_r$, respectively, to $[x'_{r-i+j} \ 0]$. In either case, removing the row $x_{r-i+2}$ or $1_r$ from the resulting matrix gives $F'_i$, as seen in Definition~\ref{Fi'def1}. This proves the second and the third cases of part (ii). \qed \par
See section 5.1.1 of \cite{kimWhit} for an example of row operations.

\section{Interlude: Combinatorics of the birational map}\label{sec:comb}
Before moving on to the column operations, let us derive decomposition formulas of path-polynomials, making use of their combinatorial properties. For this, we introduce the notion of \textit{products} of collections of sets of disjoint paths. For Definitions~\ref{def:calPprod} through~\ref{def:uPprod}, we let $O_1$, $O_2$, $D_1$ and $D_2$ be sets of nonzero entries of $wn$ such that $O_1$ and $O_2$ are disjoint, $D_1$ and $D_2$ are disjoint, $\abs{O_1}=\abs{D_1}$ and $\abs{O_2}=\abs{D_2}$. 
\begin{definition}\label{def:calPprod}
   For any subcollections ${\mathcal P}_1\subseteq {\mathcal P}(O_1\to D_1)$ and ${\mathcal P}_2\subseteq {\mathcal P}(O_2\to D_2)$, we define the product ${\mathcal P}_1 \otimes {\mathcal P}_2$ as
\[\begin{aligned}
 &{\mathcal P}_1\otimes {\mathcal P}_2 =\Bigl\{\{p_1,\ldots,p_a,q_1,\ldots,q_b\} \mid  \{p_1,\ldots,p_a\}\in {\mathcal P}_1, \, \{q_1,\ldots,q_b\}\in {\mathcal P}_2\Bigr\}.
\end{aligned}\] 
\end{definition}
Observe that the paths in a set $\mathbf{p}\cup \mathbf{q}\in {\mathcal P}_1\otimes {\mathcal P}_2$ are not necessarily disjoint, because some $p_i\in \mathbf{p}$ and $q_j\in\mathbf{q}$ may intersect.
\begin{definition}\label{def:calPprodintersect}
    Let ${\mathcal P}_1\subseteq {\mathcal P}(O_1\to D_1)$ and ${\mathcal P}_2\subseteq {\mathcal P}(O_2\to D_2)$. We use the subscript $\ooalign{$\swarrow$\cr\hfil$\nwarrow$\hfil}$ to denote the subcollection of the product ${\mathcal P}_1 \otimes {\mathcal P}_2$ consisting of the sets that contain paths that intersect. That is,
\[
 ({\mathcal P}_1 \otimes {\mathcal P}_2){\crarrow} =\bigl\{\{p_1,\ldots,p_a,q_1,\ldots,q_b\}\in {\mathcal P}_1 \otimes {\mathcal P}_2\mid \text{some } p_i \text{ and } q_j \text{ intersect }\bigr\}.\]
\end{definition}
Recall Definition~\ref{def:upath}. We extend the notion of $u(\mathbf{p})$ to the sets in the product ${\mathcal P}_1 \otimes {\mathcal P}_2$.
\begin{definition}\label{def:uPprod}
For $\mathbf{p}\in {\mathcal P}(O_1\to D_1)$ and $\mathbf{q}\in {\mathcal P}(O_2\to D_2)$, we let $u(\mathbf{p}\cup \mathbf{q})=u(\mathbf{p}) u(\mathbf{q})$.
\end{definition}
If $\mathbf p$ and $\mathbf q$ are disjoint, that is, if any $p\in \mathbf p$ and $q\in \mathbf q$ are disjoint, then $\mathbf p\cup \mathbf q$ is a set of disjoint paths in ${\mathcal P}(O_1\cup O_2\to D_1\cup D_2)$, and the above definition agrees with Definition~\ref{def:upath}. If $\mathbf p$ and $\mathbf q$ are not disjoint, then some $p\in \mathbf p$ and $q\in \mathbf q$ contain one or more common elements. Suppose that $n_\a\in V_w$ is both in $p$ and $q$. Since $\mathbf p$ and $\mathbf q$ are sets of disjoint paths, $n_\a$ does not appear in any of the paths in $\mathbf p$ other than $p$ and any of the paths in $\mathbf q$ other than $q$. Hence $n_\a$ appears exactly twice in the set $\mathbf p\cup \mathbf q$, once in the path $p$ and once in the path $q$, so the degree of $n_\a$ in $u(\mathbf p\cup \mathbf q)$ is two. From the above definition, it is clear that
\begin{equation}\label{uPprodsum}
\begin{aligned}
    \sum\limits_{\mathbf{p}\cup \mathbf{q}\in {\mathcal P}_1\otimes {\mathcal P}_2} u(\mathbf p \cup \mathbf q)&=\Bigl(\sum\limits_{\mathbf p\in {\mathcal P}_1}u(\mathbf p) \Bigr)\Bigl(\sum\limits_{\mathbf q\in {\mathcal P}_2} u(\mathbf q)\Bigr), 
\end{aligned}
\end{equation}
where the left hand side is summed over $\mathbf{p}\in {\mathcal P}_1$ and $\mathbf{q}\in {\mathcal P}_2$. 
\begin{definition}\label{PLstarsum}
Let $A$ and $B$ be sets of nonzero entries of $wn$ such that $A\cap B=\emptyset$ and $\abs{A}=\abs{B}$. We define
\[
  P_L^{\star}(A\to B)=\begin{cases}
  P_L(A\to B) &\text{if}\quad \abs{A}=\abs{B}\geq 1, \\
  1 &\text{if}\quad A=B=\emptyset.
  \end{cases}  
\]
\end{definition}

The following lemma gives a decomposition formula for $P_1(n_\alpha)$ defined in Definition~\ref{P1def}.
\begin{lem}[Lemma 4.25 of \cite{kimWhit}]\label{P1lem}
Let $\npdj \in V_w$. Write
$O_1(\npdj)=\{n_{\pi(\d_1),r}, \ldots 
, n_{\pi(\d_s),r}\}$, $D_1(n_{\pi(d),j})=\{1_{\pi(\delta_1)},\ldots,1_{\pi(\delta_s)}\}$. We have 
\begin{equation}\label{P1expand}
\begin{aligned}
P_1(n_{\pi(d),j})=\sum_{m=1}^{s} \  &P\bigl(\{n_{\pi(\delta_m-1),r}\}\cup O_{\uparrow,\delta_m-1}(n_{\pi(d),j})\to D_{\uparrow,\delta_m-1}(n_{\pi(d),j})\bigr) \\
&\times n_{\pi(\delta_m),r}\cdot P_L\bigl(O_{\downarrow,\delta_m+1}(n_{\pi(d),j})\to D_{\downarrow,\delta_m}(n_{\pi(d),j})\bigr).
\end{aligned}\end{equation}
Here, if $O_1(\npdj)=D_1(\npdj)=\emptyset$, then $s=0$ and the sum equals $0$.
\end{lem}
\begin{proof} By the remark below Definition~\ref{P2def}, the sum $P_1(\npdj)$ equals zero in case $O_1(\npdj)$ is empty. Hence, we shall assume that $s\geq 1$. For $1\leq m\leq s$, let ${\mathcal P}^{(m)}(\npdj)$ denote the subcollection of ${\mathcal P}_1(\npdj)$ consisting of the sets in which the path starting from $n_{\pi(\d_m),r}$ starts by going up, and every path whose origin is below the $\d_m$-th row starts by going left. Then we have the partition
\[{\mathcal P}_1(\npdj)=\bigsqcup_{1\leq m\leq s}  {\mathcal P}^{(m)}(\npdj).\]
Let $P^{(m)}(\npdj)=\sum\limits_{\textbf{p}\in {\mathcal P}^{(m)}(\npdj)} u(\textbf{p})$. We have $P_1(\npdj)=\sum\limits_{1\leq m\leq s} P^{(m)}(\npdj)$,
so the lemma follows once we show that
\begin{equation}\label{Pmexp}
\begin{aligned}
P^{(m)}(n_{\pi(d),j})= &n_{\pi(\delta_m),r}\cdot P\bigl(\{n_{\pi(\delta_m-1),r}\}\cup O_{\uparrow,\delta_m-1}(n_{\pi(d),j})\to D_{\uparrow,\delta_m-1}(n_{\pi(d),j})\bigr) \\
&\quad\times P_L\bigl(O_{\downarrow,\delta_m+1}(n_{\pi(d),j})\to D_{\downarrow,\delta_m}(n_{\pi(d),j})\bigr).
\end{aligned}\end{equation}
Consider $\mathbf{p}\in {\mathcal P}^{(m)}(n_{\pi(d),j})$. The path $p_0\in \mathbf{p}$ whose origin is $n_{\pi(\d_m),r}$ can be regarded as a path starting form $n_{\pi(\d_m-1),r}$ to a destination above the $(\d_m-1)$-th row, prepended by $n_{\pi(\d_m),r}$. The $P$-polynomial multiplied by $n_{\pi(d_m),r}$ in the first line of \eqref{Pmexp} represent the path $p_0$ and the paths whose origin is above $n_{\pi(\d_m),r}$, and the $P_L$-polynomial in the second line represent the paths whose origin is below the $\d_m$-th row, all of which start by going left, by the construction of ${\mathcal P}^{(m)}(n_{\pi(d),j})$. This proves $\eqref{Pmexp}$.
\end{proof}
A detailed proof of \eqref{Pmexp} can be found in \cite[Lemma 4.25]{kimWhit}. Next, we derive a formula for $R_1(\npdj)$ from Lemma~\ref{P1lem}. We use the following lemma.
\begin{lem}\label{lem:DupOup}
    Let $n_{\pi(d),j}\in V_w$, and let $a'$ satisfy $1_{\pi(a')}\in D_1(n_{\pi(d),j})$. Then $n_{\pi(a'-1),j}$ is either a free variable or equal to $1_j$, and we have
\[
D_{\uparrow,a'-1}(n_{\pi(d),j})=
\begin{cases}
D(n_{\pi(a'-1),j}) & \text{if } n_{\pi(a'-1),j}\in V_w \\
\{1_j\} & \text{if } n_{\pi(a'-1),j}=1_j,
\end{cases}
\]
and
\[
\{n_{\pi(a'-1),r}\}\cup O_{\uparrow,a'-1}(n_{\pi(d),j})=
\begin{cases}
O(n_{\pi(a'-1),j}) & \text{if } n_{\pi(a'-1),j}\in V_w \\
\{ n_{j,r} \} & \text{if } n_{\pi(a'-1),j}=1_j.
\end{cases}
\]
\end{lem}
\begin{proof}
    If $1_{\pi(a')}\in D_1(n_{\pi(d),j})$, then by Definition~\ref{D1}, $1_{\pi(a')}$ is not the top destination $1_j$ of $n_{\pi(d),j}$, thus $\pi\i(j)<a'\leq d$. If $a'-1=\pi\i(j)$ then $n_{\pi(a'-1),j}=1_j$, and if $a'-1>\pi\i(j)$ then the fact that $n_{\pi(a'-1),r}$ is not an origin of $n_{\pi(d),j}$ implies that $1_{\pi(a'-1)}$ is to the left of the $j$-th column, hence $\pi(a'-1)<j$. It follows that $n_{\pi(a'-1),j}\in V_w$. In either case, the lemma is straightforward.
\end{proof}

\begin{cor}\label{R1lem}
With the same hypotheses given in Lemma~\ref{P1lem}, the rational function $R_1(n_{\pi(d),j})$ satisfies 
\[
R_1(n_{\pi(d),j})=\sum_{m=1}^{s} (-1)^{t_m} \cdot n_{\pi(\delta_m),r} \cdot u_{\pi(\d_m-1),j}\cdot \frac{P_L\bigl(O_{\downarrow,\delta_m+1}(n_{\pi(d),j})\to D_{\downarrow,\delta_m}(n_{\pi(d),j})\bigr)}{\prod\limits_{1_\mu\in D_{\downarrow,\delta_m}(n_{\pi(d),j})}\rho(1_\mu)},
\]
where $t_{m}=\abs{D_{\downarrow,\delta_m}(n_{\pi(d),j})}$.
\end{cor}
\begin{proof}
    Since $D_{\uparrow,\delta_m-1}(n_{\pi(d),j})$ and $D_{\downarrow,\delta_m}(n_{\pi(d),j})$ partition $D(\npdj)$, we have
\begin{equation}\label{P1npdjprod}
\prod\limits_{1_\mu\in D(\npdj)}\rho(1_\mu)=\prod\limits_{1_\mu\in D_{\uparrow,\delta_m-1}(n_{\pi(d),j})}\rho(1_\mu)\cdot \prod\limits_{1_\mu\in D_{\downarrow,\delta_m}(n_{\pi(d),j})}\rho(1_\mu). 
\end{equation}
Let $t_{\pi(d),j}=\abs{D(\npdj)}-1$ and $t_{\pi(\d_m-1),j}=\abs{D_{\uparrow,\delta_m-1}(n_{\pi(d),j})}-1$, so that we have $t_{\pi(d),j}=t_{\pi(\d_m-1),j}+t_m$.
Dividing both sides of \eqref{P1expand} by $(-1)^{t_{\pi(d),j}}\prod\limits_{1_\mu\in D(\npdj)}\rho(1_\mu)$ and then applying Lemma~\ref{lem:DupOup} and \eqref{P1npdjprod} yields the corollary.
\end{proof}
Our next lemma describes a decomposition formula for the polynomial $P_L(O\to D)$. Fix a Weyl group element $w\in \Omega$ of $\GL(r)$.
\begin{definition}\label{def:PLdecomp}
Let $D=\{1_{l(1)}, 1_{l(2)}, \ldots,  1_{l(m)}\}$, $l(1)<l(2)<\cdots<l(m)$ be a set of $1$-entries in $wn$, listed in the order as they appear in the matrix from left to right. For a positive integer $d\leq r$, we say that a set $D$ of $1$-entries in $wn$ is ``$P_L$-decomposable set with respect to $d$" if $\pi\i\bigl(l(j)\bigr)<d$ and $\pi(d)<l(j)$ for all $1\leq j\leq m$, that is, if every element of $D$ is above and to the right of $1_{\pi(d)}$.
\end{definition}
\noindent For the next definition, recall Definition~\ref{def:gamma1}.
\begin{definition}\label{def:PLdecomporigin}
Let $D=\{1_{l(1)}, 1_{l(2)}, \ldots,  1_{l(m)}\}$ be a $P_L$ decomposition set with respect to $d$. Let $1\leq h\leq m$ be the positive integer such that $1_{l(h)}$ is the element of $D$ in the highest row among all the rows containing an element of $D$. We define the set $O$ of ``origins for a $P_L$-decomposable set $D$ with respect to $d$'' to be the set
\begin{equation}\label{def:PLdecompO}
\{\gamma(1_{\pi(d)})\}\cup \gamma(D)\smallsetminus \{\gamma(1_{l(h)})\},\end{equation}  
that is, $O$ consists of $\gamma(1_{\pi(d)})$, and the rightmost entries of the rows containing an element of $D$ other than $1_{l(h)}$, the highest element of $D$.
\end{definition}
Since $\pi(d)\neq l(j)$ for all $1\leq j\leq m$, $O$ consists of $m$ elements, the same as $D$. A $P_L$-decomposable set $D$ with respect to $d$ and its corresponding set $O$ are defined in a way that mimics the configuration of $D(n_\a)$ and $O(n_\a)$ for a free variable $n_\a$ of $wn$: the lowest element $\gamma(1_{\pi(d)})$ of $O$ is below every element of $D$, the highest element $1_{l(h)}$ of $D$ is above every element of $O$, and each element of $O\smallsetminus \{\gamma(1_{\pi(d)})\}$ is paired with an element of $D\smallsetminus \{1_{l(h)}\}$ in the same row.\par 
As in Definition~\ref{ODup}, we denote $O_{\downarrow, k}$ and $D_{\downarrow,k}$ the subsets of $O$ and $D$ consisting of the elements which are in or below the $k$-th row of $wn$, respectively, and $O_{\uparrow, k}$ and $D_{\uparrow, k}$ the subsets of $O$ and $D$ consisting of the elements which are in or above the $k$-th row of $wn$, respectively. The following lemma is the Lemma 4.45 of \cite{kimWhit} with minor modifications.
\begin{lem}\label{lem:firststep}
Let $D=\{1_{l(1)},1_{l(2)},\ldots,1_{l(m)}\}$ be a $P_L$-decomposable set with respect to a positive integer $d$ with $1\leq d\leq r$ and let $O$ defined as in \eqref{def:PLdecompO}. If the leftmost entry of $D$ is not the highest entry of $D$, that is, if $h>1$, then
\begin{equation}\label{firstexp}
\begin{aligned}
P_L(O\to D)&= P_L\bigl(O_{\downarrow,\pi^{-1}(l(1))+1}\to D_{\downarrow,\pi^{-1}(l(1))}\bigr) \cdot  P_L\bigl(O_{\uparrow,\pi^{-1}(l(1))}\to D_{\uparrow,\pi^{-1}(l(1))-1}\bigr) \\ 
 &\quad\quad\quad\quad\quad\quad-\rho(1_{l(1)})\cdot P_L\bigl(O\smallsetminus \{\gamma(1_{l(1)})\}\to D\smallsetminus \{1_{l(1)}\}\bigr).
\end{aligned}
\end{equation}
\end{lem}
\begin{proof} First, we claim that every set in ${\mathcal P}_L(O\to D)$ is contained in the product
\begin{equation}\label{calPLprod}{\mathcal P}_L(O_{\downarrow,\pi^{-1}(l(1))+1}\to D_{\downarrow,\pi^{-1}(l(1))}) \otimes {\mathcal P}_L(O_{\uparrow,\pi^{-1}(l(1))}\to D_{\uparrow,\pi^{-1}(l(1))-1}).\end{equation}
If $\mathbf{p}\in {\mathcal P}_L(O\to D)$ contains a horizontal path $p_0$ connecting $1_{l(1)}$ and $\gamma(1_{l(1)})$, then any path $p\in \mathbf{p}$ whose origin is below $\gamma(1_{l(1)})$ cannot have its destination above $1_{l(1)}$, because it forces $p$ to go across $p_0$, because $1_{l(1)}$ is located in the leftmost column among every element of $D$. An elementary counting argument shows that this is impossible, hence the path in $\mathbf{p}$ starting from $\gamma(1_{l(1)})$ must end above the $\pi\i(l(1))$-th row. From this, the claim follows. \par 
However, the product \eqref{calPLprod} is not necessarily equal to ${\mathcal P}_L(O\to D)$, because a set in the first collection and a set in the second collection may not be disjoint. We must exclude the sets in the product of the two collections which contain two paths that intersect to obtain ${\mathcal P}_L(O\to D)$. To simplify notations, denote the first and the second collections in the product \eqref{calPLprod} ${\mathcal P}_L(\downarrow,1)$ and ${\mathcal P}_L(\uparrow,1)$, respectively, and let
$P_L(\downarrow,1)=\sum_{\mathbf p\in {\mathcal P}_L(\downarrow,1)}u(\mathbf p)$ and $P_L(\uparrow,1)=\sum_{\mathbf p\in {\mathcal P}_L(\uparrow,1)}u(\mathbf p)$. We have
\begin{equation}\label{auxlemint}
 {\mathcal P}_L(O\to D) =({\mathcal P}_L(\downarrow,1) \otimes {\mathcal P}_L(\uparrow,1))\smallsetminus ({\mathcal P}_L(\downarrow,1) \otimes {\mathcal P}_L(\uparrow,1))\crarrow.
\end{equation}
Since $1_{l(h)}$ is above and to the right of $1_{l(1)}$, we have $n_{l(1),l(h)}\in V_w$. Hence, by Definition~\ref{def:gamma1}, $\gamma(1_{l(b)})=n_{l(1),b}\in V_w$ for some $b\geq l(h)$. Let
${\mathcal A}=({\mathcal P}_L(\downarrow,1) \otimes {\mathcal P}_L(\uparrow,1))\crarrow$ and \[\mathcal{B}= {\mathcal P}_L(O\smallsetminus \{n_{l(1),b}
\}\to D\smallsetminus \{1_{l(1)}\})\otimes {\mathcal P}_L (\{n_{l(1),b}\}\to \{1_{l(1)}\}).\]
We claim that the sets in $\mathcal A$ are in one-to-one correspondence with the sets in $\mathcal B$. First, observe that since $n_{l(1),b}$ and $1_{l(1)}$ are in the same row, there is only one way of connecting $n_{l(1),b}$ and $1_{l(1)}$, which is to proceed horizontally to the left. Hence ${\mathcal P}_L\bigl(\{n_{l(1),b}\}\to \{1_{l(1)}\}\bigr)$ contains a singleton set, whose element is the horizontal path. Next, let $\mathbf p=\{p_1,\ldots,p_a\}\in {\mathcal P}_L(\downarrow,1)$, $\mathbf q=\{q_1,\ldots,q_b\}\in {\mathcal P}_L(\uparrow,1)$, and suppose that
\begin{equation}\label{eq:pqintersect}
\mathbf p\cup \mathbf q=\{p_1,\ldots,p_a,q_1,\ldots,q_b\}\in \bigl({\mathcal P}_L(\downarrow,1) \otimes {\mathcal P}_L(\uparrow,1)\bigr)\crarrow,\end{equation}
that is, the two sets $\mathbf p$ and $\mathbf q$ are not disjoint. Since $1_{l(1)}$ is the highest destination of ${\mathcal P}_L(\downarrow,1)$, there is exactly one path in $\mathbf p$ which passes through any of the nonzero entries of $\pi\i(l(1))$-th row, and it is the path whose destination is $1_{l(1)}$. All the other paths in $\mathbf p$ stay below the $\pi\i(l(1))$-th row. Similarly, there is exactly one path in $\mathbf q$ which passes through any of the nonzero entries of $\pi\i(l(1))$-th row, and it is the path whose origin is $n_{l(1),b}$. All other paths in $\mathbf q$ stay above the $\pi\i(l(1))$-th row. Therefore, any intersection of a path in $\mathbf p$ and a path in $\mathbf q$ must occur in the $\pi\i(l(1))$-th row of $wn$. Since there is exactly one path in $\mathbf{p}$ and one path in $\mathbf q$ which go through the $\pi\i(l(1))$-th row, they are the only two paths that intersect. \par 
Let $p_i\in \mathbf p$ and $q_j\in \mathbf q$ be these two intersecting paths. Then the destination of $p_i$ is $1_{l(1)}$, and the origin of $q_j$ is $n_{l(1),b}$. The two paths intersect at one or more nonzero entries in the $\pi\i(l(1))$-th row. Recall that $1_{l(1)}$ is the leftmost element of $D$. Recall that $\mathbf q$ is a set in ${\mathcal P}_L(\uparrow,1)$. Since the destination of $q_j$ is above the $(\pi\i(l(1)))$-th row, $p_i$ and $q_j$ do not intersect at $1_{l(1)}$. Hence $p_i$ and $q_j$ intersect at free variables. Let $c$ be the greatest positive integer such that $n_{l(1),c}\in V_w$ and $n_{l(1),c}$ is both in $p_i$ and $q_j$, that is, $n_{l(1),c}$ is the rightmost free variable where $p_i$ and $q_j$ intersect. We split $p_i$ into two subpaths, one from the origin of $p_i$ to $n_{l(1),c}$, and the other from $n_{l(1),c}$ to the destination $1_{l(1)}$. We call these two subpaths of $p_i$ the ``head'' and ``tail'' of $p_i$, respectively. Similarly, we split $q_j$ into two subpaths, one from $n_{l(1),b}$ to $n_{l(1),c}$, and the other from $n_{l(1),c}$ to the destination of $q_j$. We call these subpaths the ``head'' and the ``tail'' of $q_j$, respectively.
\par
We interchange the tails of $p_i$ and $q_j$ to form two new paths. The path which consists of the head of $q_j$ and the tail of $p_i$ starts from $n_{l(1),b}$, proceeds horizontally to $n_{l(1),c}$, and then again horizontally to $1_{l(1)}$. Hence this is the unique horizontal path which connects $n_{l(1),b}$ to $1_{l(1)}$. The other path, which consists of the head of $p_i$ and the tail of $q_j$ starts from the origin of $p_i$, proceeds to $n_{l(1),c}$, and then reaches the destination of $q_j$. Denote the two new paths by $p^\star$ and $s_{ij}$, respectively.  Then $\{p^\star\}$ is in ${\mathcal P}_L(\{n_{l(1),b}\}\to \{1_{l(1)}\})$. Let 
    $\mathbf{h}=\{p_1,\ldots,p_{i-1},p_{i+1},\ldots,p_a,q_1,\ldots,q_{j-1},q_{j+1},\ldots,q_b,s_{ij}\}$. Then $\mathbf{h}$ is in 
${\mathcal P}_L\bigl(O\smallsetminus \{n_{l(1),b}
\}\to D\smallsetminus \{1_{l(1)}\}\bigr)$, consequently
$\mathbf{h}\cup \{p^\star\}$ is in $\mathcal B$. This establishes the correspondence from $\mathcal A$ to $\mathcal B$. \par
The other direction of correspondence can be shown similarly. Let $\mathbf s=\{s_1,\ldots,s_k\}$ be a set in ${\mathcal P}_L(O\smallsetminus \{n_{l(1),b}
\}\to D\smallsetminus \{1_{l(1)}\})$, and let $p^\star$ be the path from $n_{l(1),b}$ to $1_{l(1)}$, which is necessarily horizontal. As before, it is straightforward to verify that there is a unique path $s_t$ that intersects with $p^\star$. We can fix a point of intersection $n_{l(1),c}\in V_w$, with $c$ maximal, and swap the tails of $s_t$ and the tails of $p^\star$ exactly as before, and replacing $s_t$ and $p^\star$ with these two new paths. We can regroup the paths into two sets, so that their union is in the product $\bigl({\mathcal P}_L(\downarrow,1) \otimes {\mathcal P}_L(\uparrow,1)\bigr)\crarrow$. This establishes the correspondence from $\mathcal B$ to $\mathcal A$. Since both directions are given by tail-swapping and the choice of tails is unique due to the maximality of $c$, the two constructions are inverse to each other. This establishes the one-to-one correspondence between $\mathcal A$ and $\mathcal B$. \par
If $\mathbf p\cup \mathbf q$ in $\mathcal A$ and 
$\mathbf s \cup \{p^\star\}$ in $\mathcal B$ correspond to each other, then
$u(\mathbf p\cup\mathbf q)=u(\mathbf s\cup\{p^\star\})=u(\mathbf s)\rho(1_{l(1)})$.
It follows that the sum 
$\sum\limits_{\mathbf p\cup \mathbf q\in \bigl({\mathcal P}_L(\downarrow,1) \otimes {\mathcal P}_L(\uparrow,1)\bigr)\crarrow} u(\mathbf p\cup \mathbf q)$, where $\mathbf p$ and $\mathbf q$ are elements of ${\mathcal P}_L(\downarrow,1)$ and $ {\mathcal P}_L(\uparrow,1)$, respectively, equals $P_L\bigl(O\smallsetminus \{n_{l(1),b}
\}\to D\smallsetminus \{1_{l(1)}\}\bigr) \cdot \rho(1_{l(1)})$. Together with \eqref{uPprodsum} and \eqref{auxlemint}, this yields \eqref{firstexp}.
\end{proof}
\begin{cor}[Lemma 4.75 of \cite{kimWhit}]\label{auxlem}
If $D$ is a $P_L$ decomposable set with respect to $d$ defined in Definition~\ref{def:PLdecomp} and $O$ is the set defined as in (\ref{def:PLdecompO}), the polynomial $P_L(O\to D)$ equals the sum
\begin{equation}\label{altexpansion}
    \begin{aligned}
&\sum_{q=1}^{h} \Bigr((-1)^{q-1}\cdot \prod_{1\leq q_1<q} \rho(1_{l(q_1)}) \\ &\quad\quad\quad \times P_L\bigl(O_{\downarrow, \pi^{-1}\left(l(q)\right)+1}\smallsetminus \bigcup_{1\leq q_1<q}\{\gamma(1_{l(q_1)})\}\to D_{\downarrow, \pi^{-1}\left(l(q)\right)}\smallsetminus \bigcup_{1\leq q_1<q}\{1_{l(q_1)}\}\bigr) \\
&\quad\quad\quad \times P_L^\star \bigl(O_{\uparrow, \pi^{-1}\left(l(q)\right)}\smallsetminus \bigcup_{1\leq q_1<q}\{\gamma(1_{l(q_1)})\}\to D_{\uparrow, \pi^{-1}\left(l(q)\right)-1}\smallsetminus \bigcup_{1\leq q_1<q}\{1_{l(q_1)}\}\bigr)\Bigr).
    \end{aligned}
\end{equation}
\end{cor}
The formula \eqref{altexpansion} can be obtained simply by an iteration of Lemma~\ref{lem:firststep}. The proof is provided in Appendix~\ref{sec:technical}.

\section{Column operations on $T'_i$ and $F'_i$}\label{sec:colop}
Throughout the section, we assume that $r\neq \pi(r)$ and fix a positive integer $i$ that satisfies $\pi(r)<r-i+1$, so that the bottom row of $T_i$ and $F_i$ does not contain $1_{\pi(r)}$. Recall from Definition~\ref{x'kdef} that the row operations described in the proof of Proposition~\ref{prop:rowop} reduce the entries of $T_i$ and $F_i$ of the form $u_\a=R_L(n_\a)+R_1(n_\a)+R_2(n_\a)$ with $n_\a\in V_w$ to $R_L(n_\a)+R_1(n_\a)$. However, the row operations turn some zero entries of $T_i$ and $F_i$ to nonzero, as seen in the second case of \eqref{x'kdef1}. We apply a series of column operations to $T'_i$ and $F'_i$ to turn these entries back to zero and further reduce the entries $R_L(n_\a)+R_1(n_\a)$ with $n_\alpha\in V_w$ to $R_L(n_\a)$. By Definitions~\ref{xLkdef} and~\ref{TLFLdef}, the resulting matrices are $T^L_i$ and $F^L_i$, respectively. 
\begin{definition}\label{TLFLcoldef}
We let $C'_{r-i+j}$, $D'_{r-i+j}$, $C^L_{r-i+j}$ and $D^L_{r-i+j}$ denote the $j$-th columns of $T'_i$, $F'_i$, $T^L_i$, and $F^L_i$, respectively. We write the matrices as \[\begin{aligned} T'_i&=[C'_{r-i+1} \ C'_{r-i+2} \ \ \ldots\ \ C'_{r-1}],\quad 
    F'_i=[D'_{r-i+1} \ D'_{r-i+2} \ \ \ldots\ \ D'_{r-1}], \\
    T^L_i&=[C^L_{r-i+1} \ C^L_{r-i+2} \ \ \ldots\ \ C^L_{r-1}],\quad \text{and}\quad
    F^L_i=[D^L_{r-i+1} \ D^L_{r-i+2} \ \ \ldots\ \ D^L_{r-1}].\end{aligned}\]
\end{definition}
We apply column operations on $T'_i$ and $F'_i$, starting from the rightmost column. We then proceed to the left, reducing each column $C'_j$ to $C^L_j$ using the columns $C^L_l$ with $l>j$.
\begin{lem}\label{Qjllembase}
For any $d\geq r-i$, $d\neq \pi\i(r)$, we have $x'_d(r-1)=x^L_d(r-1)$.
\end{lem}
\begin{proof}
The entry $x_d(r-1)$ is either $0$, $1_{r-1}$, or $R(n_{\pi(d),r-1})$ with $n_{\pi(d),r-1}\in V_w$. Since $d\neq \pi\i(r)$, the first case $x_d(r-1)=0$ implies that $x_d(r-1)$ is above $1_{r-1}$. It follows that if $x_d(r-1)$ is either $0$ or $1_{r-1}$, then the entries above $x_d(r-1)$ are all zero. By Definition~\ref{x'kdef} and Definition~\ref{xLkdef}, we conclude that $x'_d(r-1)=x^L_d(r-1)$ if $x_d(r-1)$ is either $0$ or $1_{r-1}$. \par
Now, suppose that $x_d(r-1)=u_{\pi(d),r-1}$, where $n_{\pi(d),r-1}$ is a free variable in $V_w$. Then $1_{r-1}$ is the top destination of $n_{\pi(d),r-1}$, and $n_{\pi(d),r-1}$ has at most one additional destination, which is $1_r$. If $1_r\in D(n_{\pi(d),r-1})$, then we also have $\gamma(1_r)=1_r\in O(n_{\pi(d),r-1})$. Hence $1_r$ only contributes as a path of length $0$. Hence, the sets in ${\mathcal P}(n_{\pi(d),r-1})$ contain a unique path with nonzero length, which starts from the bottom origin and ends at $1_{r-1}$. By the remark below Definition~\ref{P2def}, we have $P_1(n_{\pi(d),r-1})=R_1(n_{\pi(d),r-1})=0$. The lemma follows from Definitions~\ref{x'kdef} and~\ref{xLkdef}.
\end{proof}
The following proposition describes the column operations. Recall that $r-i+1>\pi(r)$.
\begin{prop}\label{Qjllem}
For each pair of integers $(j,l)$ satisfying $r-i+1\leq j<l\leq r-1$, there exists a rational function $Q_j(l)\in \Z\left[\left\{n_\a,n\i_\a\mid n_\a\in V_w\right\}\right]$ that satisfies
\begin{equation}\label{Qjlrole}
x'_d(j)+\sum_{j+1\leq l\leq r-1} Q_j(l)\cdot x^L_d(l)=x^L_d(j)
\end{equation}
for every $d\geq r-i$, $d\neq \pi\i(r)$.
\end{prop}
The proof of this proposition is deferred to section~\ref{rowcolopproof}. Let us deduce Proposition~\ref{prop:colop} from
Proposition~\ref{Qjllem}. It is a direct consequence of Proposition~\ref{Qjllem} that we have
\[    C'_j+Q_j(j+1)\cdot C^L_{j+1}+\cdots+Q_j(r-1)\cdot C^L_{r-1}=C^L_j
\]
and the same holds with $C$ replaced by $D$. By Lemma~\ref{Qjllembase}, the rightmost column of $T'_i$ satisfies $C'_{r-1}=C^L_{r-1}$. By \eqref{Qjlrole} with $j=r-2$, adding $Q_{r-2}(r-1)\cdot C^L_{r-1}$ to $C'_{r-2}$ reduces $C'_{r-2}$ to $C^L_{r-2}$. After that, we add $Q_{r-3}(r-2)\cdot C^L_{r-2}+Q_{r-3}(r-1)\cdot C^L_{r-1}$ to $C'_{r-3}$ to reduce the column to $C^L_{r-3}$. Proceeding left until we reduce every column $C'_j$ with $r-i+1\leq j\leq r-1$ to $C^L_j$, we obtain $T_i^L$. By the same process with $C$ replaced by $D$, we can reduce $F'_i$ to $F^L_i$. Each step of this column operation can be represented by a lower triangular unipotent matrix, hence Proposition~\ref{prop:colop} follows.

\subsection{Formula for the rational functions $Q_j(l)$}\label{Qjlsec}
In this section, we give the formula for the rational functions $Q_j(l)$ that satisfies \eqref{Qjlrole}. For Definitions~\ref{Djl0} through~\ref{def:colfm}, we let $j$ satisfy $2\leq j\leq r-2$ and $n_{\pi(r),j}\in V_w$.
\begin{definition}\label{Djl0}
For $1\leq l\leq r$, we let $D^{(0)}_{j,l}=D_1(n_{\pi(r),j})\cap D_{\nearrow,l}(n_{\pi(r),j})$.
\end{definition}

\begin{definition}\label{Djlk2}
For $l\leq r$ and $k\leq r$ such that $n_{l,r}\in V_w$ and $n_{\pi(k),r}\in V_w$, we let
\[D^{(2)}_{j,l,k}=\bigl(D_{\nearrow,l}(n_{\pi (r),j})\smallsetminus \{1_l\}\bigr)\cap D_{\downarrow,k}(n_{\pi(r),j}) \eqand O^{(2)}_{j,l,k}=\{n_{l,r}\}\cup \gamma( D^{(2)}_{j,l,k})\smallsetminus\{n_{\pi(k),r}\}.\]
\end{definition}
\begin{definition}\label{signjkl}
We let $\e^{(2)}_{j,l,k}=\lvert \djlktwo \rvert$.
\end{definition}
\begin{definition}\label{def:colfm}
For $1\leq l\leq r$, we let
\begin{equation}\label{colfm} \begin{aligned}
    Q_j(l) &=\sum\limits_{1_{\pi(k)}\in D^{(0)}_{j,l}} \biggl( (-1)^{\e^{(2)}_{j,l,k}}\cdot n_{\pi(k),r} \cdot u_{\pi(k-1),j}\cdot\frac{P_L^{\star}\bigl(O^{(2)}_{j,l,k}\to D^{(2)}_{j,l,k}\bigr)}{\prod\limits_{1_\mu\in  D^{(2)}_{j,l,k}}\rho(1_\mu)}\biggr).
\end{aligned}
\end{equation}
\end{definition}
To verify that the formula for $Q_j(l)$ is well-defined, we need the following lemmas.
\begin{lem}\label{lem:Djlz}
    Let $2\leq j\leq r-2$ satisfy $n_{\pi(r),j}\in V_w$ and $l\leq r$. If $\djlz\neq \emptyset$, then $n_{l,r}\in V_w$. Also, if $1_{\pi(k)}\in D_{j,l}^{(0)}$ then $n_{\pi(k),r}\in V_w$.
\end{lem}
\begin{proof}
Observe that if $\djlz\neq \emptyset$ then $1_l$ is below $1_r$, since by Definition~\ref{D1} every element of $D_1(\nprj)$ is below the row containing $1_r$. It follows that $n_{l,r}\in V_w$. Also, by the remark below Definition~\ref{O1}, if $1_{\pi(k)}\in D_1(\nprj)$ then $n_{\pi(k),r}\in V_w$. This proves the lemma.
\end{proof}

\begin{lem}\label{lem:ODsetsize}
    Let $2\leq j\leq r-2$ satisfy $n_{\pi(r),j}\in V_w$, $l\leq r$ satisfy $\djlz\neq \emptyset$, and let $k\leq r$ satisfy $1_{\pi(k)}\in D_{j,l}^{(0)}$. Then
   $ \lvert O^{(2)}_{j,l,k}\rvert =\lvert D^{(2)}_{j,l,k}\rvert$.
\end{lem}
\begin{proof}
If $l=\pi(k)$ then $1_l$ is in the $k$-th row, hence by Definition~\ref{Djlk2} we have $\djlktwo=\ojlktwo=\emptyset$. If $l\neq \pi(k)$, then $n_{l,r}\notin \gamma (\djlktwo)$ since $1_l\notin \djlktwo$, and $n_{\pi(k),r}\in \gamma (\djlktwo)$ since $1_{\pi(k)}\in \djlz\subseteq D_{\nearrow,l}(\nprj)$. From this, the lemma follows.
\end{proof}
 Lemma~\ref{lem:Djlz} implies that if $D^{(0)}_{j,l}$ is nonempty and $1_{\pi(k)}\in \djlz$ then $n_{l,r}$ and $n_{\pi(k),r}$ are free variables, hence the sets $D^{(2)}_{j,l,k}$ and $O^{(2)}_{j,l,k}$ are well-defined. Also, by Lemma~\ref{lem:ODsetsize},  $\lvert O^{(2)}_{j,l,k}\rvert =\lvert D^{(2)}_{j,l,k}\rvert$ whenever $\djlz\neq \emptyset$ and $1_{\pi(k)}\in \djlz$. This verifies that the formula (\ref{colfm}) for $Q_j(l)$ is well-defined. See section 5.4.1 of \cite{kimWhit} for an example illustrating the column operations.

\subsection{Four types of the entries of $wu$}\label{fourtypes}
Observe that by Definition~\ref{xLkdef} we have $x^L_d(l)=0$ whenever $u_{\pi(d),l}=0$, hence we can rewrite \eqref{Qjlrole} as
\begin{equation}\label{Qjlrole1} x'_d(j)+\sum_{\substack{j+1\leq l\leq r-1  \\ u_{\pi(d),l}\neq 0}} Q_j(l)\cdot x^L_d(l)=x^L_d(j).\end{equation} 
Let $j$ and $d$ satisfy $r-i+1\leq j\leq r-2$, $r-i\leq d$ and $d\neq \pi\i(r)$. We separate the entries $x_d(j)$ into four types, depending on which of the following $d$ and $j$ satisfy:
\begin{enumerate}
    \item $d\leq \max\bigl(\pi\i(j),\pi\i(r)-1\bigr)$.
    \item $d>\max\bigl(\pi\i(j),\pi\i(r)\bigr)$ and $\pi(d)<j$.
    \item $d>\max\bigl(\pi\i(j),\pi\i(r)\bigr)$, and $\pi(d-1)\leq j<\pi(d)$.
    \item $d>\max\bigl(\pi\i(j),\pi\i(r)\bigr)$, and $j<\min\bigl(\pi(d-1),\pi(d)\bigr)$.
\end{enumerate}
It is clear that any $x_d(j)$ falls into precisely one of the four categories. \par If $x_d(j)$ belongs to type (1), then the inequality implies that $x_d(j)$ is either above the $\pi\i(r)$-th row, or in or above the $\pi\i(j)$-th row. If $x_d(j)$ is above the $\pi\i(r)$-th row, then by Lemma~\ref{rowopcor} we have $x_d(j)=x'_d(j)$, and by the remark below Definition~\ref{def:calpleftn} we have $x_d(j)=x^L_d(j)$. On the other hand, if $x_d(j)$ is in or above the $\pi\i(j)$-th row, then we have $x_{d'}(j)=0$ for all $d'<d$, since they are above the entry $1_j$. Therefore, $x_d(j)$ belongs to the third case of \eqref{x'kdef1} and the second case of \eqref{xLkdef1}, hence $x'_d(j)=x^L_d(j)=x_d(j)$. Thus if $x_d(j)$ belongs to type (1), \eqref{Qjlrole1} is equivalent to
\begin{equation}\label{type1goal}
 u_{\pi(d),j}+\sum_{\substack{j+1\leq l\leq r-1  \\ u_{\pi(d),l}\neq 0}}Q_j(l)\cdot x^L_d(l)= u_{\pi(d),j}.
\end{equation}
If $x_d(j)$ belongs to type (2), then we have $d>\pi\i(j)$ and $\pi(d)<j$, so $n_{\pi(d),j}\in V_w$. In this case, we have $x'_d(j)=R_L(n_{\pi(d),j})+R_1(n_{\pi(d),j})$ and $x^L_d(j)=R_L(n_{\pi(d),j})$, by Definitions~\ref{x'kdef} and~\ref{xLkdef}. In this case, \eqref{Qjlrole1} is equivalent to
\begin{equation}\label{type2goal}
R_L(n_{\pi(d),j})+R_1(n_{\pi(d),j})+\sum_{\substack{j+1\leq l\leq r-1  \\ u_{\pi(d),l}\neq 0}}Q_j(l)\cdot x^L_d(l)=R_L(n_{\pi(d),j}). 
\end{equation}
If $x_d(j)$ belongs to type (3), then the above inequalities imply that $x_d(j)=0$ and $x_{d-1}(j)\neq 0$, so $x_d(j)$ belongs to the second case of the definition \eqref{x'kdef1}. Hence $x'_d(j)=-n_{\pi(d),r}\cdot u_{\pi(d-1),j}$. Also, by Definition~\ref{xLkdef}, we have $x^L_d(j)=0$. Hence \eqref{Qjlrole1} is equivalent to
\begin{equation}\label{type3goal}
 -n_{\pi(d),r}\cdot u_{\pi(d-1),j}+\sum_{\substack{j+1\leq l\leq r-1  \\ u_{\pi(d),l}\neq 0}}Q_j(l)\cdot x^L_d(l)=0.    
\end{equation}
Finally, if $x_d(j)$ is of type (4), then $x'_d(j)=x^L_d(j)=0$. In this case,  \eqref{Qjlrole1} is equivalent to
\begin{equation}\label{type4goal}
    \sum_{\substack{j+1\leq l\leq r-1  \\u_{\pi(d),l}\neq 0}}Q_j(l)\cdot x^L_d(l)=0.
\end{equation}
We prove Proposition~\ref{Qjllem} by showing that rational functions $Q_j(l)$ defined in (\ref{colfm}) satisfy the equations (\ref{type1goal}), (\ref{type2goal}), (\ref{type3goal}), and (\ref{type4goal}) depending on the type of $x_d(j)$.

\subsection{Proof of Proposition~\ref{Qjllem} }\label{rowcolopproof}
As discussed at the beginning of section~\ref{sec:colop}, Proposition~\ref{Qjllem} implies Proposition~\ref{prop:colop}. Together with Proposition~\ref{prop:rowop}, this completes the proof of Proposition~\ref{thm:rowcolop}. Recall that $i$ is a fixed integer that satisfies $r-i+1>\pi(r)$. Fix an integer $j$ satisfying 
\begin{equation}\label{eq:jcond}
    \pi(r)<r-i+1\leq j\leq r-2
\end{equation} and an integer $d$ satisfying \begin{equation}\label{eq:dcond}
    d\geq r-i, \quad d\neq \pi\i(r).
\end{equation} 
Note that \eqref{eq:jcond} implies that $n_{\pi(r),j}\in V_w$, hence the definitions given in section~\ref{Qjlsec} are well defined. \par
The proof of Proposition~\ref{Qjllem} uses the decomposition lemmas given in section~\ref{sec:comb}. We apply Corollary~\ref{auxlem} to the polynomial $P_L\bigl(O^{(2)}_{j,l,k}\to D^{(2)}_{j,l,k}\bigr)$, which appears in the formula \eqref{colfm} for $Q_j(l)$. This turns the sum in \eqref{Qjlrole} into a double sum, and we deduce Proposition~\ref{Qjllem} by rearranging the sum. The precise statement is the following.
 \begin{lem}\label{lem:qjlsumnew}
    Let $2\leq j\leq r-2$ satisfy $\nprj\in V_w$ and $d$ satisfy $d>\max\bigl(\pi\i(j),\pi\i(r)\bigr)$. Assume that $D^{(0)}_{j,\pi(d)}$ is nonempty. We have
    \begin{equation}\label{eq:qjlsumnew}
    Q_j\bigl(\pi(d)\bigr)+ \sum\limits_{\substack{\max(j,\pi(d))< l\leq r-1  \\ u_{\pi(d),l}\neq 0}} Q_j(l) R_L(\npdl)=\begin{cases}
        n_{\pi(d),r}u_{\pi(d-1),j} & \text{if } 1_{\pi(d)}\in D^{(0)}_{j,\pi(d)} \\
        0 & \text{otherwise}.
    \end{cases}   
    \end{equation}
\end{lem}
We also use the following lemma. 
\begin{lem}\label{lem:npdj}
 Let $2\leq j\leq r-2$ satisfy $n_{\pi(r),j}\in V_w$, and let $1\leq d\leq r$ satisfy $\npdj\in V_w$. We have $D_{\uparrow,d}(\nprj)=D(\npdj)$ and $D_1(\npdj)=D^{(0)}_{j,\pi(d)}$.
\end{lem}
For the sake of brevity, we postpone the proof of these two lemmas to Appendix~\ref{sec:technical}. We are now ready to prove Proposition~\ref{Qjllem}.
\begin{lem}\label{Qjltype1}
   Suppose that $x_d(j)=u_{\pi(d),j}$ belongs to type (1), that is, $j$ and $d$ satisfy the inequality $d\leq \max\bigl(\pi\i(j),\pi\i(r)-1\bigr)$. Then the rational functions $Q_j(l)$ defined in \eqref{colfm} satisfy \eqref{type1goal}.
\end{lem}
\begin{proof}
Let $l$ satisfy $j<l\leq r-1$ and $u_{\pi(d),l}\neq 0$. Then we have $\pi\i(l)\leq d$, thus the inequality $d\leq \max\bigl(\pi\i(j),\pi\i(r)-1\bigr)$ implies that $1_l$ is either above $1_j$ or above $1_r$. If $1_l$ is above $1_j$ then $D_{\nearrow,l}(\nprj)$ is empty. On the other hand, if $1_l$ is above $1_r$, then since every element of the set $D_1(\nprj)$ is below the row containing $1_r$, the intersection $D_1(\nprj)\cap D_{\nearrow,l}(\nprj)$ is empty. By Definitions~\ref{Djl0} and~\ref{def:colfm}, we conclude that $Q_j(l)=0$. From this \eqref{type1goal} follows.
\end{proof}

\begin{lem}\label{Qjltype2}
   Suppose that $x_d(j)=u_{\pi(d),j}$ belongs to type (2), that is, $j$ and $d$ satisfy
   \begin{equation}\label{Type2ineq}
   d>\max\bigl(\pi\i(j),\pi\i(r)\bigr), \quad \text{and}\quad \pi(d)<j.  
   \end{equation} 
   Then the rational functions $Q_j(l)$ satisfy \eqref{type2goal}.
\end{lem}
\begin{proof}
The inequality \eqref{Type2ineq} implies that $n_{\pi(d),j}$ is a free variable below the row containing $1_r$. If $l>j$ satisfies $u_{\pi(d),l}\neq 0$, then the inequality $\pi(d)<j<l$ implies that $n_{\pi(d),l}\in V_w$. By Definition~\ref{xLkdef}, we have $x^L_d(l)=R_L(n_{\pi(d),l})$. Thus we can rewrite (\ref{type2goal}) as
\begin{equation}\label{type2goal2}
 R_1(n_{\pi(d),j})+\sum_{\substack{j+1\leq l\leq r-1  \\ u_{\pi(d),l}\neq 0}}Q_j(l) R_L(n_{\pi(d),l})=0. 
\end{equation}
By \eqref{eq:jcond} and \eqref{Type2ineq}, $d$ and $j$ satisfy the hypotheses for Lemmas~\ref{lem:qjlsumnew} and~\ref{lem:npdj}:
\[2\leq j\leq r-2, \quad  \nprj\in V_w, \quad \npdj\in V_w, \eqand d>\max\bigl(\pi\i(j),\pi\i(r)\bigr).\]
By Lemma~\ref{lem:npdj}, we have $D^{(0)}_{j,\pi(d)}=D_1(n_{\pi(d),j})$. Also, since $\pi(d)<j$, $1_{\pi(d)}$ is to the left of $1_j$, hence $1_{\pi(d)}\notin D_1(n_{\pi(d),j})$. Therefore, by Lemma~\ref{lem:qjlsumnew}, the lemma follows once we show that $Q_j\bigl(\pi(d)\bigr)=R_1(\npdj)$. Write
$O_1(\npdj)=\{n_{\pi(\b_1),r},\ldots,n_{\pi(\b_s),r}\}$ and $D_1(n_{\pi(d),j})=\{1_{\pi(\b_1)},1_{\pi(\b_2)},\ldots,1_{\pi(\b_s)}\}$, with $s=0$ in case $O_1(\npdj)=D_1(\npdj)=\emptyset$. We have
\begin{equation}\label{qjlsum2}
Q_{j}\left(\pi(d)\right)=\sum_{m=1}^{s} \biggl( (-1)^{\e^{(2)}_{j,\pi(d),\b_m}}\cdot n_{\pi(\b_m),r} \cdot u_{\pi(\b_m-1),j}\cdot \frac{ P_L^{\star}\bigl(O^{(2)}_{j,\pi(d),\b_m}\to D^{(2)}_{j,\pi(d),\b_m}\bigr)}{\prod_{1_\mu\in D^{(2)}_{j,\pi(d),\b_m}}\rho(1_\mu)}\biggr).
\end{equation}
With Definition~\ref{Djlk2}, it is straightforward to verify that $D_{\downarrow,\b_m}(\npdj)=D^{(2)}_{j,\pi(d),\b_m}$ and $O_{\downarrow,\b_m+1}(\npdj)=O^{(2)}_{j,\pi(d),\b_m}$ (see below \cite[(5.132)]{kimWhit} for a proof of this fact). By Lemma~\ref{R1lem}, we conclude that $Q_j\bigl(\pi(d)\bigr)=R_1(\npdj)$. This finishes the proof.
\end{proof}

%type3starts

\begin{lem}\label{Qjltype3}
Suppose that $x_d(j)=u_{\pi(d),j}$ belongs to type (3), so that we have
\begin{equation}\label{Type3ineq}
    d>\max\bigl(\pi\i(j),\pi\i(r)\bigr) \quad\text{and}\quad \pi(d-1)\leq j<\pi(d).
\end{equation} 
Then the rational functions $Q_j(l)$ satisfy the equation \eqref{type3goal}.
\end{lem}
\begin{proof} 
The inequality \eqref{Type3ineq} implies that $\pi(d)\neq r$, hence $j+1\leq \pi(d)\leq r-1$. If $l=\pi(d)$ then by Definition~\ref{xLkdef} we have $x^L_d(l)=x_d(l)=1_{\pi(d)}$. Also, if $l<\pi(d)$ then $u_{\pi(d),l}=0$, and if $l>\pi(d)$ and $u_{\pi(d),l}\neq 0$ then $n_{\pi(d),l}\in V_w$. Hence we can rewrite \eqref{type3goal} as
\begin{equation}\label{type3goal2}
-n_{\pi(d),r} \cdot u_{\pi(d-1),j}+Q_j(\pi(d))+\sum_{\substack{\pi(d)< l\leq r-1  \\ u_{\pi(d),l}\neq 0}}Q_j(l) R_L(n_{\pi(d),l})=0.
\end{equation}
By the same approach we used for the previous lemma, \eqref{type3goal2} follows from Lemma~\ref{lem:qjlsumnew} once we show that $1_{\pi(d)}\in D_{j,\pi(d)}^{(0)}$. Recall from \eqref{Type3ineq} that we have $d>\pi\i(j)$ and $\pi(d)>j$, hence $1_{\pi(d)}$ is below and to the right of $1_j$. It follows that $1_{\pi(d)}$ is a destination of $\nprj$, and by Definition~\ref{Dne} we have $1_{\pi(d)}\in D_{\nearrow,\pi(d)}(\nprj)$. \par 
By Definition~\ref{Djl0}, it remains to show that $1_{\pi(d)}\in D_1(\nprj)$. For this, by Definition~\ref{D1}, we need to show that $n_{\pi(d),r}\in O_1(\nprj)$. It follows from \eqref{Type3ineq} that $n_{\pi(d),r}\in V_w$, thus $n_{\pi(d),r}=\gamma(1_{\pi(d)})$ is an origin of $\nprj$. It follows immediately from \eqref{Type3ineq} that $\pi(d-1)< r$, hence $n_{\pi(d-1),r}\in V_w$. Also, the inequality $\pi(d-1)\leq j$ in \eqref{Type3ineq} implies that the entry $1_{\pi(d-1)}$ is either equal to $1_j$ or to the left of $1_j$. If $1_{\pi(d-1)}=1_j$ then $1_{\pi(d-1)}$ is the top destination of $\nprj$, and it follows from Definition~\ref{Odef} that $n_{\pi(d-1),r}$ is not an origin of $\nprj$. On the other hand, if $1_{\pi(d-1)}$ is to the left of $1_j$, then $1_{\pi(d-1)}$ is not in $D(\nprj)$, consequently $n_{\pi(d-1),r}$ is not an origin of $\nprj$. By Definition~\ref{O1}, we conclude that $n_{\pi(d),r}\in O_1(\nprj)$.
\end{proof}

\begin{lem}\label{Qjltype4}
Suppose that $x_d(j)=u_{\pi(d),j}$ belongs to type (4), so that we have
\begin{equation}\label{Type4ineq} d>\max\bigl(\pi\i(j),\pi\i(r)\bigr) \quad\text{and}\quad j<\min\bigl(\pi(d-1),\pi(d)\bigr).
\end{equation} Then the rational functions $Q_j(l)$ satisfy \eqref{type4goal}.
\end{lem}
\begin{proof}
Just like the previous case of type (3), we can rewrite the equation \eqref{type4goal} as
\begin{equation}\label{type4goal2}
 Q_j(\pi(d))+\sum_{\substack{\pi(d)< l\leq r-1  \\ u_{\pi(d),l}\neq 0}}Q_j(l)  R_L(n_{\pi(d),l})=0.   
\end{equation}
Again by Lemma~\ref{lem:qjlsumnew}, it suffices to show that $1_{\pi(d)}\notin D^{(0)}_{j,\pi(d)}$ to prove \eqref{type4goal2}. It follows from the inequalities $d>\pi\i(j)$ and $j<\pi(d-1)$ in (\ref{Type4ineq}) $1_{\pi(d-1)}$ is below and to the right of $1_j$, hence is an element of $D(\nprj)$. It also follows from \eqref{Type4ineq} that $d-1\geq \pi\i(r)$, so $1_{\pi(d-1)}$ is either equal to $1_r$ or below $1_r$. In either case, $n_{\pi(d-1),r}$ is nonzero. Hence by Definition~\ref{Odef}, $n_{\pi(d-1),r}\in O(\nprj)$. This shows that $1_{\pi(d)}\notin D_1(\nprj)$. By Definition~\ref{Djl0}, we conclude that $1_{\pi(d)}\notin D^{(0)}_{j,\pi(d)}$. This finishes the proof.
\end{proof}
At this point, we have covered all of the four possible types of $x_d(j)$. This finishes the proof of Proposition~\ref{Qjllem}.

\appendix
\section{Proof of technical lemmas}\label{sec:technical}
\subsection{Proof of Corollary~\ref{RLwRwtilde}} We give here the proof of the last case of \eqref{RLwRwtilde0}, in which the top destination $1_b$ of $n_{a,b}$ is below $1_r$. We omit the proof of the first two cases because they are routine and use the same logic. By Definition~\ref{rhodef}, we have $\rho^{(w)}(1_\mu)=\rho^{(\wtilde)}(1_\mu)n_{\mu,r}$ if $1_\mu$ is below $1_r$, and $\rho^{(w)}(1_\mu)=\rho^{(\wtilde)}(1_\mu)$ if $1_\mu$ is above $1_r$. Recall Definition~\ref{def:RLR1R2}. Using superscript $(w)$, write
    \[R_L^{(w)}(n_{a,b})=(-1)^{t^{(w)}_{a,b}}
    \frac{P_L^{(w)}(n_{a,b})}{\prod\limits_{1_\mu\in D^{(w)}(n_{a,b})}\rho^{(w)}(1_\mu)}, \quad t^{(w)}_{a,b}=\lvert D^{(w)}(n_{a,b}) \rvert -1.\]
    Since $1_{r}\notin D^{(w)}(n_{a,b})$, we have $D^{(w)}_{a,b}=D^{(\wtilde)}_{a,b}$ and $t^{(w)}_{a,b}=t^{(\wtilde)}_{a,b}$. Let 
    \[S_1=\prod\limits_{1_\mu\in D^{(w)}(n_{a,b})} \rho^{(w)}(1_\mu) \eqand S_2=\prod\limits_{n_{\b,r}\in O^{(w)}_{\downarrow, \pi^{-1}(r)+1}(n_{a,b})} n_{\b,r}=\prod\limits_{n_{\b,r}\in O^{(w)}(n_{a,b})}n_{\b,r}.\] 
    The second equality of $S_2$ follows from the assumption that $1_b$ is below $1_r$.
    The bottom origin $n_{a,r}$ of $n_{a,b}$ is not in the product $S_1$ but is in the product $S_2$. On the other hand, $n_{b,r}$, the rightmost element of the row containing the top destination $1_b$ of $n_{a,b}$, is in the product $S_1$ but is not in the product $S_2$. If $\mu$ satisfies $\mu\neq b$ and $1_{\mu}\in D^{(w)}(n_{a,b})$, then $n_{\mu,r}$ is both in $S_1$ and $S_2$, and satisfy $\rho^{(w)}(1_\mu)=\rho^{(\wtilde)}(1_{\mu})n_{\mu,r}$. From \eqref{PLwPwtilde} we deduce that $R^{(w)}_L(n_{a,b})$ equals
    \[R^{(w)}_L(n_{a,b})=(-1)^{t^{(w)}_{a,b}}\frac{S_2}{S_1}P^{(\wtilde)}(n_{a,b})=\frac{n_{a,r}}{n_{b,r}}(-1)^{t^{(\wtilde)}_{a,b}}\frac{P^{(\wtilde)}(n_{a,b})}{\prod\limits_{1_\mu\in D^{(\wtilde)}(n_{a,b})}\rho^{(\wtilde)}(1_\mu)}.\]
    This finishes the proof. \qed
\subsection{Proof of Corollary~\ref{auxlem}}
 To simplify notations, for $1\leq q\leq h$ let $P_L(\downarrow,q)$ denote the $P_L$-polynomial in the second line of \eqref{altexpansion}, and $P_L(\uparrow,q)$ denote the $P_L$-polynomial in the last line of \eqref{altexpansion} without the superscript $\star$. Also, write 
$O^{(q)}=O\smallsetminus \bigcup_{1\leq q_1<q}\{\gamma(1_{l(q_1)})\}$ and $D^{(q)}=D\smallsetminus \bigcup_{1\leq q_1<q}\{1_{l(q_1)}\}$.
Observe that $O^{(1)}=O$, $D^{(1)}=D$, $O^{(j)}\smallsetminus \{n_{l(j),r}\}=O^{(j+1)}$, and $D^{(j)}\smallsetminus \{1_{l(j)}\}=D^{(j+1)}$. For $1\leq j\leq h$, let $S(j)$ be the function defined by
\begin{equation}\label{Sjdef}
    \begin{aligned}
&\sum_{q=1}^{j-1} \Bigl(\prod_{1\leq q_1<q} \rho(1_{l(q_1)})\Bigr) (-1)^{q-1}P_L(\downarrow,q)P_L(\uparrow,q)+\Bigl(\prod_{1\leq q_1<j} \rho(1_{l(q_1)})\Bigr) (-1)^{j-1} P_L(O^{(j)}\to D^{(j)}).
    \end{aligned}
\end{equation}
Observe that $S(1)$ equals $P_L(O\to D)$ and $S(h)$ equals the sum \eqref{altexpansion}. Hence the corollary follows once we show that $S(j)=S(j+1)$ for all $j<h$. To show this, we apply Lemma~\ref{lem:firststep} to the polynomial $P_L(O^{(j)}\to D^{(j)})$. This gives \begin{equation}\label{PLjexp}
\begin{aligned}
    &P_L(O^{(j)}\to D^{(j)}) =P_L(\downarrow,j)\cdot P_L(\uparrow,j)-\rho(1_{l(j)})\cdot P_L(O^{(j+1)}\to D^{(j+1)}).
\end{aligned}
\end{equation}
Replacing $P_L(O^{(j)}\to D^{(j)})$ in $S(j)$ with the right hand side of \eqref{PLjexp} gives $S(j+1)$. \qed

\subsection{Lemmas from section~\ref{rowcolopproof}} Our final task is to prove Lemmas~\ref{lem:qjlsumnew} and~\ref{lem:npdj}. We first prove Lemma~\ref{lem:npdj}.
\begin{proof}[Proof of Lemma~\ref{lem:npdj}]
The first equation $D_{\uparrow,d}(\nprj)=D(\npdj)$ follows directly from Definition~\ref{mwna}. For the second equation, recall Definition~\ref{Djl0}. Suppose that $1_{\mu}\in D_{j,\pi(d)}^{(0)}$. Since $1_{\mu}$ is in or above the $d$-th row and is an element of $D(\nprj)$, we have $1_{\mu}\in D(\npdj)$. Together with the fact that $1_\mu\in D_1(\nprj)$, this implies $1_\mu\in D_1(\npdj)$. Thus $D_{j,\pi(d)}^{(0)}\subseteq D_1(\npdj)$. Conversely, suppose that $1_{\mu}\in D_1(\npdj)$. Then it is clear that $1_{\mu}\in D_1(\nprj)$. Also, since $1_{\mu}$ is in $D(\npdj)$, $1_{\mu}$ is above the $d$-th row of $wn$, and in or to the right of $j$-th column of $wn$. Since $\pi(d)<j$, $1_{\mu}$ is to the right of the $\pi(d)$-th column. This shows that $1_{\mu}\in D_{\nearrow,\pi(d)}(\nprj)$, thus $1_{\mu}\in D^{(0)}_{j,\pi(d)}$. We conclude that $D_1(\npdj)=D^{(0)}_{j,\pi(d)}$.
\end{proof}
Lemma~\ref{lem:qjlsumnew} requires several lemmas. Recall Definition~\ref{Djlk2} of the sets $\djlktwo$ and $\ojlktwo$.
\begin{lem}\label{lem:PLO2D2}
    Let $2\leq j\leq r-2$ satisfy $n_{\pi(r),j}\in V_w$, $l\leq r$ satisfy $n_{l,r}\in V_w$, and $k\leq r$ satisfy \begin{equation}\label{eq:kcondition}
        1_{\pi(k)}
    \in D_{\nearrow,l}(\nprj)\smallsetminus \{1_j,1_l\}\quad \text{and}\quad n_{\pi(k),r}\in V_w.
    \end{equation}
We have the following:
    \begin{enumerate}
        \item[(i)] The entry $1_{\pi(k)}$ is an element of the set $\djlktwo$, and is in the highest row among all the elements of $\djlktwo$.
        \item[(ii)] Let $m=\ejlktwo=\lvert\djlktwo\rvert$. Write $\djlktwo=\{1_{\l(1)},\ldots,1_{\l(m)}\}$
    with $\l(1)<\cdots<\l(m)$, and let $1\leq h\leq m$ satisfy $\l(h)=\pi(k)$.
    We have
   \begin{equation}\label{eq:PLO2D2exp}
\begin{aligned}
   & (-1)^{\ejlktwo} \frac{P_L\bigl(O^{(2)}_{j,l,k}\to D^{(2)}_{j,l,k}\bigr)}{\prod\limits_{1_\mu\in  D^{(2)}_{j,l,k}}\rho(1_\mu)}=-\sum_{q=1}^{h}\ \biggl((-1)^{\e^{(2)}_{j,\l(q),k}} \frac{P^{\star}_L\bigl(O^{(2)}_{j,\l(q),k}\to D^{(2)}_{j,\l(q),k}\bigr)}{\prod\limits_{1_\mu\in  D^{(2)}_{j,\l(q),k}}\rho(1_\mu)} R_L\left(n_{l,\l(q)}\right)\biggr).
    \end{aligned}
\end{equation}
\end{enumerate}
\end{lem}
\begin{proof} Observe that \eqref{eq:kcondition} implies $1_{\pi(k)}\in D_{\downarrow,k}(\nprj)$. Part (i) is a direct consequence of this fact and Definition~\ref{Djlk2}. For part (ii), observe that the set $\djlktwo$ is a $P_L$-decomposable set with respect to $d=\pi\i(l)$ defined in Definition~\ref{def:PLdecomp} because every element of $\djlktwo$ is above and to the right of $1_l$. Also, $\ojlktwo$ defined in Definition~\ref{Djlk2} agrees with the definition given in \eqref{def:PLdecompO}, by part (i) of this lemma. Observe that since $n_{\pi(k),r}\in V_w$, every element of $\ojlktwo$, which is below the $k$-th row of $wn$, is also a free variable. By Corollary~\ref{auxlem}, the polynomial
$P_L(\ojlktwo\to\djlktwo)$ equals
\begin{equation}\label{eq:PLO2D2exp1}
\begin{aligned}
    &\sum_{q=1}^{h} \Bigl(\prod_{1\leq q_1<q} \rho(1_{\l(q_1)})\Bigr)\cdot (-1)^{q-1} \cdot P_L\bigl(O(\downarrow,q)\to D(\downarrow,q)\bigr) \cdot P_L^{\star} \bigl(O(\uparrow,q) \to D(\uparrow,q)\bigr),
\end{aligned}\end{equation}
where
\begingroup
\allowdisplaybreaks
\begin{align*}
    O(\downarrow,q)&=\{n_{x,r}\in \ojlktwo \mid \pi\i(x)\geq \pi\i\bigl(\l(q)\bigr)+1\}\smallsetminus \bigcup_{1\leq q_1<q}\{\gamma(1_{\l(q_1)})\},  \\
D(\downarrow,q)&=\{1_{\l(b)}\in \djlktwo \mid \pi\i\bigl(\l(b)\bigr)\geq \pi\i\bigl(\l(q)\bigr)\}\smallsetminus \bigcup_{1\leq q_1<q}\{1_{\l(q_1)}\},  \\
O(\uparrow,q)&=\{n_{x,r}\in \ojlktwo \mid \pi\i(x)\leq \pi\i\bigl(\l(q)\bigr)\}\smallsetminus \bigcup_{1\leq q_1<q}\{\gamma(1_{\l(q_1)})\},  \\
D(\uparrow,q)&=\{1_{\l(b)}\in \djlktwo \mid \pi\i\bigl(\l(b)\bigr)\leq \pi\i\bigl(\l(q)\bigr)-1\}\smallsetminus \bigcup_{1\leq q_1<q}\{1_{\l(q_1)}\}.
\end{align*}
\endgroup
With Definition~\ref{Djlk2}, it is straightforward to check that we have
\[D(\downarrow,q) =D(n_{l,\l(q)}), \quad O(\downarrow,q) =O(n_{l,\l(q)}), \quad D(\uparrow,q)=D^{(2)}_{j,\l(q),k}, \quad O(\uparrow,q)=O^{(2)}_{j,\l(q),k} \]
for each $1\leq q\leq h$ (see below \cite[(5.98)]{kimWhit} for a detailed proof). This gives
\begin{equation}\label{eq:PLO2D2exp2}
\begin{aligned}
    P_L (O^{(2)}_{j,l,k}\to D^{(2)}_{j,l,k}) &=\sum_{q=1}^{h}\Bigl(\prod_{1\leq q_1<q}\rho(1_{\l(q_1)})\Bigr) (-1)^{q-1}  P_L(n_{l,\l(q)})P^{\star}_L(O^{(2)}_{j,\l(q),k}\to D^{(2)}_{j,\l(q),k}).
\end{aligned}
\end{equation}
Finally, to obtain the formula (\ref{eq:PLO2D2exp}), we divide both sides of (\ref{eq:PLO2D2exp2}) by
\begin{equation}\label{dividingfactor}
    (-1)^{\e^{(2)}_{j,l,k}}\prod\limits_{1_\mu\in  D^{(2)}_{j,l,k}}\rho(1_\mu).
\end{equation}
Since the sets $D(\downarrow,q)$, $D(\uparrow,q)$, and $\bigcup_{1\leq q_1 <q} \{1_{\lambda(q_1)}\}$ with $1\leq q\leq h$ partition $\djlktwo$, we have
\[
\djlktwo= D^{(2)}_{j,\l(q),k}\sqcup D(n_{l,\l(q)})\sqcup\bigcup_{1\leq q_1<q}\{1_{\l(q_1)}\}.\]
This implies $\e^{(2)}_{j,l,k}=\e^{(2)}_{j,\l(q),k}+\lvert D(n_{l,\l(q)}) \rvert+q-1$,
and
\begin{equation}\label{eq:D2rhoprod}
    \prod\limits_{1_{\mu}\in D^{(2)}_{j,l,k}}\rho(1_\mu)= \prod\limits_{1_{\mu}\in D^{(2)}_{j,\l(q),k}}\rho(1_\mu)\cdot\prod\limits_{1_{\mu}\in D(n_{l,\l(q)})}\rho(1_\mu)\cdot\prod\limits_{1\leq q_1<q}\rho(1_{\l(q_1)}).
\end{equation}
Let $t_{l,\l(q)}=\lvert D(n_{l,\l(q)}) \rvert-1$. Dividing both sides of (\ref{eq:PLO2D2exp2}) by the product (\ref{dividingfactor}) and then applying \eqref{eq:D2rhoprod} gives \eqref{eq:PLO2D2exp}.
\end{proof}

\begin{lem}\label{lem:Djpdzindex}
Let $j$ and $d$ satisfy $2\leq j\leq r-2$, $\nprj\in V_w$, $d>\max\bigl(\pi\i(j),\pi\i(r)\bigr)$, and $D^{(0)}_{j,\pi(d)}\neq \emptyset$. Write
$D^{(0)}_{j,\pi(d)}=\left\{1_{\pi(\b_1)},1_{\pi(\b_2)},\ldots,1_{\pi(\b_s)}\right\}$, $\b_1<\b_2<\cdots<\b_s$.
\begin{enumerate}
    \item [(i)] We have $n_{\pi(d),r}\in V_w$, and each $1\leq m\leq s$ satisfies $n_{\pi(\b_m),r}\in V_w$. In particular, the set $D^{(2)}_{j,\pi(d),\b_m}$ is well-defined for each $1\leq m\leq s$.
    \item [(ii)] If $\b_m=d$ for some $1\leq m\leq s$, then $m$ necessarily equals $s$, and $D^{(2)}_{j,\pi(d),\b_m}=\emptyset$.
    \item [(iii)] For each $1\leq m\leq s$ such that $\b_m\neq d$, the integers $j$, $\pi(d)$ and $\b_m$ satisfy the hypotheses for $j$, $l$, and $k$ given in Lemma~\ref{lem:PLO2D2}, that is, $j$, $\pi(d)$ and $\b_m$ satisfy $n_{\pi(d),r}\in V_w$, $1_{\pi(\b_m)}
    \in D_{\nearrow,\pi(d)}(\nprj)\smallsetminus \{1_j,1_{\pi(d)}\}$, and $n_{\pi(\b_m),r}\in V_w$.
    \item [(iv)] If $1\leq m\leq s$ satisfies $\b_m\neq d$, then $1_{\pi(\b_m)}\in D^{(2)}_{j,\pi(d),\b_m}$, and $1_{\pi(\b_m)}$ is the element of $D^{(2)}_{j,\pi(d),\b_m}$ which is located in the highest row among all the elements.
    \end{enumerate}
    \end{lem}
\begin{proof}
The inequality $d>\pi\i(r)$ implies $n_{\pi(d),r}\in V_w$. Also, it follows from Definition~\ref{Djl0} that $1_{\pi(\b_m)}\in D_1(\nprj)$, hence by the remark below Definition~\ref{O1} we have $n_{\pi(\b_m),r}\in V_w$. This proves (i). For (ii), observe that every element of $D_{j,\pi(d)}^{(0)}$ is in or above the $d$-th row, hence if $1_{\pi(\b_m)}=1_{\pi(d)}$ then $1_{\pi(\b_m)}$ must be the element that is located in the lowest row among all elements, which implies $m=s$. Also, observe from Definition~\ref{Djlk2} that if $1_{\mu}\in D^{(2)}_{j,\pi(d),\b_m}$ then $1_{\mu}$ is above the $d$-th row of $wn$, and in or below the $\b_m$-th row. Therefore, if $\b_m=d$ then $D^{(2)}_{j,\pi(d),\b_m}$ is an empty set. This proves (ii). For part (iii), observe that Definitions~\ref{O1} and~\ref{D1} imply that $1_{\pi(\b_m)}$ is below the row containing $1_j$, and that Definition~\ref{Djl0} implies that $1_{\pi(\b_m)}$ is an element of $D_{\nearrow,\pi(d)}(\nprj)$. Together with part (i), this proves part (iii). Part (iv) is a direct consequence of part (iii) of this lemma and part (i) of Lemma~\ref{lem:PLO2D2}.
\end{proof}

\begin{lem}\label{lem:sumrearrange}
Let $j$ and $d$ satisfy $2\leq j\leq r-2$, $\nprj\in V_w$, $d>\max\bigl(\pi\i(j),\pi\i(r)\bigr)$, and $D^{(0)}_{j,\pi(d)}\neq \emptyset$. Write
$D^{(0)}_{j,\pi(d)}=\{1_{\pi(\b_1)},1_{\pi(\b_2)},\ldots,1_{\pi(\b_s)}\}$, $\b_1<\b_2<\cdots<\b_s$. For each $1\leq m\leq s$ such that $\b_m\neq d$, let $c_m=\lvert D^{(2)}_{j,\pi(d),\b_m}\rvert$ and write
    \[
D^{(2)}_{j,\pi(d),\b_m}=\left\{1_{\l_m(1)},1_{\l_m(2)},\ldots,1_{\l_m(c_m)}\right\}, \quad \l_m(1)<\l_m(2)<\cdots<\l_m(c_m).\]
Let $1\leq h_m\leq c_m$ be the integer that satisfies $1_{\l_m(h_m)}=1_{\pi(\b_m)}$. Let
\[\begin{aligned}
    A&=\{(a_1,a_2)\mid a_1=\b_m \text{ and } a_2=\l_m(q) \text{ for some } 1\leq m\leq s, \, \b_m\neq d,\, 1\leq q\leq h_m\}, \\
    B&=\{(b_1,b_2)\mid \max\bigl(j,\pi(d)\bigr)<b_2\leq r-1, \, u_{\pi(d),b_2}\neq 0, \, 1_{\pi(b_1)}\in D^{(0)}_{j,b_2}\}. 
\end{aligned}\]
We have $A=B$.
\end{lem}
\begin{proof}
    Let $(a_1,a_2)\in A$. By Definition~\ref{Djl0}, we have
\begin{equation}\label{eq:1pibm}
   1_{\pi(a_1)}= 1_{\pi(\b_m)}\in D_{j,\pi(d)}^{(0)}=D_1(n_{\pi(r),j})\cap D_{\nearrow,\pi(d)}(n_{\pi(r),j}),
\end{equation} 
and by Definition~\ref{Djlk2} we have
\begin{equation}\label{eq:1lmq}
    1_{a_2}=1_{\l_m(q)}\in D^{(2)}_{j,\pi(d),\b_m}=\bigl(D_{\nearrow, \pi(d)}(\nprj)\smallsetminus \{1_{\pi(d)}\}\bigr)\cap D_{\downarrow,\b_m}(\nprj).
\end{equation}
It follows from Definition~\ref{Dne} and \eqref{eq:1lmq} that $a_2>\pi(d)$ and $\pi\i(a_2)>d$.
Therefore $n_{\pi(d),a_2}\in V_w$, hence $u_{\pi(d),a_2}\neq 0$. Also, it follows from (\ref{eq:1lmq}) that $\pi\i(a_2)\geq \b_m$. Since $1_{\pi(\b_m)}$ is an element of $D_1(\nprj)$, we have $n_{\pi(\b_m),r}\in O_1(\nprj)$. Hence $1_{\pi(\b_m)}$ is below the row containing $1_j$ and the row containing $1_r$. By \eqref{eq:1lmq}, $1_{a_2}$ is below and to the right of $1_j$. Hence $a_2>j$ and $\pi\i(a_2)\geq \b_m>\pi\i(r)$. Thus $\max\bigl(j,\pi(d)\bigr)<a_2\leq r-1$. To show that $(a_1,a_2)\in B$, it remains to show that $1_{\pi(a_1)}\in D^{(0)}_{j,a_2}$. By \eqref{eq:1pibm} we have $1_{\pi(a_1)}\in D_1(\nprj)$, and by \eqref{eq:1lmq} we have $\pi\i(a_2)\geq \b_m=a_1$. Finally, we have $a_2=\l_m(q)\leq \l_m(h_m)=\pi(\b_m)=\pi(a_1)$. This shows that $1_{\pi(a_1)}\in D_{\nearrow,a_2}(n_{\pi(r),j})$. By Definition~\ref{Djl0}, and \eqref{eq:1pibm}, we conclude that $1_{\pi(a_1)}\in D^{(0)}_{j,a_2}$. This proves $A\subseteq B$. \par
Conversely, assume that $(b_1,b_2)\in B$. The conditions $\max\bigl(j,\pi(d)\bigr)<b_2$ and $u_{\pi(d),b_2}\neq 0$ together imply $d>\pi\i(b_2)$ and $\pi(d)<b_2$. This implies $D_{\nearrow,b_2}(\nprj)\subseteq D_{\nearrow,\pi(d)}(\nprj)$, consequently $D^{(0)}_{j,b_2}\subseteq D^{(0)}_{j,\pi(d)}$. Also, by Definition~\ref{Djl0}, we have $1_{\pi(b_1)}\in D_{\nearrow,b_2}(n_{\pi(r),j})$, hence $\pi(b_1)\geq b_2$ and $b_1\leq \pi\i(b_2)$. Therefore $1_{\pi(b_1)}\in D^{(0)}_{j,\pi(d)}$ and $b_1\neq \pi(d)$. Therefore $b_1=\b_m$ for some $1\leq m\leq s$, $\b_m\neq d$.
\par
Recall Definition~\ref{Djlk2}. Since $\max\bigl(j,\pi(d)\bigr)<b_2$, $1_{b_2}$ is to the right of $1_j$. Also, we have $1_{\pi(\b_m)}=1_{\pi(b_1)}\in D^{(0)}_{j,b_2}\subseteq D(\nprj)$, and this implies $\pi\i(j)\leq \b_m$ and $j\leq \pi(\b_m)$. It follows that $\pi\i(j)\leq \b_m\leq \pi\i(b_2)$. Thus $1_{b_2}$ is below and to the right of $1_j$, so $1_{b_2}\in D(\nprj)$. Together with the inequalities $d>\pi\i(b_2)$, $\pi(d)<b_2$, and $\b_m\leq \pi\i(b_2)$, we deduce
\[1_{b_2}\in \bigl(D_{\nearrow,\pi(d)}(\nprj)\smallsetminus \{1_{\pi(d)}\}\bigr)\cap D_{\downarrow,\b_m}(\nprj)=D_{j,\pi(d),\b_m}^{(2)}.\]
Hence $b_2=\l_m(q)$ for some $1\leq q\leq c_m$. Since $b_2\leq \pi(b_1) = \pi(\b_m)=\l_m(h_m)$, we have $q\leq h_m$. We conclude that $(b_1,b_2)\in A$. Therefore $A=B$.
\end{proof}
We are now ready to prove Lemma~\ref{lem:qjlsumnew}.
\begin{proof}[Proof of Lemma~\ref{lem:qjlsumnew}]
Write
\[
D^{(0)}_{j,\pi(d)}=\left\{1_{\pi(\b_1)},1_{\pi(\b_2)},\ldots,1_{\pi(\b_s)}\right\}, \quad \b_1<\b_2<\cdots<\b_s
\]
and
\begin{equation*}
D^{(0)}_{j,\pi(d)}\smallsetminus \{1_{\pi(d)}\}=\left\{1_{\pi(\b_1)},\b_2,\ldots,1_{\pi(\b_{s'})}\right\}, \quad \b_1<\b_2<\cdots<\b_{s'}.
\end{equation*}
It follows from part (ii) of Lemma~\ref{lem:Djpdzindex} that if $1_{\pi(d)}\in D^{(0)}_{j,\pi(d)}$ then we have $s'=s-1$, $\b_s=d$, $D^{(2)}_{j,\pi(d),\b_s}=\emptyset$, and $\varepsilon^{(2)}_{j,\pi(d),\b_s}=\lvert D^{(2)}_{j,\pi(d),\b_s}\rvert =0$. Otherwise, we have $s'=s$. Therefore, by Definitions~\ref{PLstarsum} and~\ref{def:colfm}, the sum
\begin{equation}\label{qjlsum1}
\sum_{m=1}^{s'} \biggl( (-1)^{\e^{(2)}_{j,\pi(d),\b_m}}\cdot n_{\pi(\b_m),r} \cdot u_{\pi(\b_m-1),j}\cdot \frac{ P_L^{\star}\bigl(O^{(2)}_{j,\pi(d),\b_m}\to D^{(2)}_{j,\pi(d),\b_m}\bigr)}{\prod\limits_{1_\mu\in D^{(2)}_{j,\pi(d),\b_m}}\rho(1_\mu)}\biggr) 
\end{equation}
equals $Q_j\left(\pi(d)\right)$ if $1_{\pi(d)} \notin D^{(0)}_{j,\pi(d)}$, and equals
$Q_j\left(\pi(d)\right)-n_{\pi(d),r}u_{\pi(d-1),j}$ if $1_{\pi(d)} \in D^{(0)}_{j,\pi(d)}$. Therefore, \eqref{eq:qjlsumnew} follows once we show that the sum \eqref{qjlsum1} equals
\[-\sum\limits_{\substack{\max(j,\pi(d))< l\leq r-1  \\ u_{\pi(d),l}\neq 0}} Q_j(l) R_L(\npdl).\]
To prove this, we use Lemma~\ref{lem:PLO2D2} to the function
\begin{equation}\label{qjlsumpart}(-1)^{\e^{(2)}_{j,\pi(d),\b_m}}\cdot\frac{P^{\star}_L\bigl(O^{(2)}_{j,\pi(d),\b_m}\to D^{(2)}_{j,\pi(d),\b_m}\bigr)}{\prod\limits_{1_{\mu}\in D^{(2)}_{j,\pi(d),\b_m}}\rho(1_\mu)}, \quad 1\leq m\leq s'.
\end{equation}
By part (iii) of Lemma~\ref{lem:Djpdzindex}, the indices $j$, $\pi(d)$ and $\b_m$ satisfy the necessary hypotheses for $j$, $l$, and $k$ given Lemma~\ref{lem:PLO2D2}, respectively, for any $1\leq m\leq s'$. Write
\[
D^{(2)}_{j,\pi(d),\b_m}=\{1_{\lambda_m (1)},\ 1_{\lambda_m(2)},\ \ldots,\ 1_{\lambda_m(c_m)}\}, \quad \lambda_m(1)<\lambda_m(2)<\cdots<\lambda_m(c_m).\]
By part (i) of Lemma~\ref{lem:PLO2D2}, $D^{(2)}_{j,\pi(d),\b_m}$ is nonempty, and $1_{\pi(\b_m)}$ is the highest element of the set $D^{(2)}_{j,\pi(d),\b_m}$. Hence we can remove the superscript $\star$ in the polynomial $P^\star_L$ in \eqref{qjlsumpart}. Let $1\leq h_m\leq c_m$ be the integer which satisfies $\lambda_m(h_m)=\pi(\b_m)$. Applying Lemma~\ref{lem:PLO2D2}, we deduce that the rational function \eqref{qjlsumpart} equals
\begin{equation}\label{eq:qjlO2D2}
\begin{aligned}
&-\sum_{q=1}^{h_m} \Bigl((-1)^{\e^{(2)}_{j,\l_m(q),\b_m}}
\cdot\frac{P_L^\star\bigl(O^{(2)}_{j,\lambda_m(q),\b_m} \to D^{(2)}_{j,\lambda_m(q),\b_m} \bigr)}{\prod\limits_{1_{\mu}\in D^{(2)}_{j,\lambda_m(q),\b_m}} \rho(1_\mu)} R_L\left(n_{\pi(d),\lambda_m(q)}\right)\Bigr).
\end{aligned}
\end{equation}
Substituting the rational function \eqref{qjlsumpart} with \eqref{eq:qjlO2D2},
we deduce that the sum \eqref{qjlsum1} equals \begin{equation}\label{qjlsumsimp}
-\sum_{m=1}^{s'} \sum_{q=1}^{h_m} \Bigl( (-1)^{\e^{(2)}_{j,\lambda_m(q),\b_m}} \cdot n_{\pi(\b_m),r} u_{\pi(\b_m-1),j} \cdot \frac{P_L^\star\bigl(O^{(2)}_{j,\lambda_m(q),\b_m} \to D^{(2)}_{j,\lambda_m(q),\b_m} \bigr)}{\prod\limits_{1_{\mu}\in D^{(2)}_{j,\lambda_m(q),\b_m}} \rho(1_\mu)} R_L\bigl(n_{\pi(d),\lambda_m(q)}\bigr)\Bigr).
\end{equation}
Next, we use Lemma~\ref{lem:sumrearrange}. Observe that \eqref{qjlsumsimp} is a sum over the elements of the set $A$ in Lemma~\ref{lem:sumrearrange}. Since $A=B$, we can rewrite \eqref{qjlsumsimp} as a sum over the elements $(k,l)$ of $B$. We deduce that \eqref{qjlsumsimp} equals
\begin{equation}\label{qjlsumsimp2}
    \begin{aligned}
&-\sum\limits_{\substack{\max(j,\pi(d)) <l\leq r-1  \\ u_{\pi(d),l}\neq 0}} \sum_{1_{\pi(k)}\in \djlz}\ \Bigl( (-1)^{\e^{(2)}_{j,l,k}}\cdot n_{\pi(k),r}u_{\pi(k-1),j} \cdot \frac{P_L^\star\bigl(O^{(2)}_{j,l,k} \to D^{(2)}_{j,l,k} \bigr)}{\prod\limits_{1_{\mu}\in D^{(2)}_{j,l,k}} \rho(1_\mu)} R_L\left(n_{\pi(d),l}\right)\Bigr).
\end{aligned}
\end{equation} 
By Definition~\ref{def:colfm}, the inner sum equals $Q_j(l)R_L(n_{\pi(d),l})$. This finishes the proof.
\end{proof}
\section{Algorithms}\label{sec:algorithm}
Fix a Weyl group element $w\in \Omega$ of $\GL(r)$. We give two algorithms that give the birational map $R=R_w$. 
\begin{subsection}{Algorithm by determinant simplification}
The following algorithm, developed by Miller in \cite{miller2008unpublished}, outputs $R(n_\a)$ for all $n_\a\in V_w$. The algorithm proceeds by simplifying the determinants of lower right subblock matrices of $wn$. \par
\underline{Step 1: $n_\a \rightsquigarrow x_\a$.}\par 
Consider each entry, going along each row from left to right, starting from the bottom row and ending at the top row.  If the entry is not one of the variables $n_\a$, do nothing.  Otherwise, select the square subblock of $A=wn$ whose bottom left entry is $n_\a$ and whose rightmost column coincides with the rightmost column of $A$; its determinant has the form $an_\a+b$ for some $a$ and $b$ which are polynomials in the variables $n_\beta$ which occur later in the sequence.  Shift the variable $n_\a$ by $n_\a\mapsto n_\a-\f ba$, so that the determinant of this block is now $(-1)^k n_\a \D$, where $k$ is the size of the subblock and $\D$ is the determinant of its $(k-1)\times (k-1)$ upper right minor.  As future $n_\beta$ which are part of this subblock are changed, $n_\a$ is also updated. At the end of this process, relabel the resulting expression as $x_\a$. \par
    \underline{Step 2: $x_\a \rightsquigarrow y_\a$.} \par Consider each entry, going downward along each column, from the rightmost column to the leftmost column of $A$.  Change the entry $x_\a$ to
\begin{equation*}
\aligned
    & x_\a \times \(\text{product of all $y_\b$ to the right of $x_\a$, in its row}\) \\
    \times & \(\text{product of all $y_\b$ above $x_\a$, in its column }\)
    \\
    \times & \(\text{product of all $y_\b$ to the right of the topmost nonzero entry above $x_\a$, in its row}\)\i.
\endaligned
\end{equation*}
As before, relabel the resulting expression as $y_\a$. \par
\underline{Step 3: $y_\a \rightsquigarrow u_\a$.} \par
Let $wy$ be the matrix obtained from $wn$ by replacing every $n_\a$ by $y_\a$. Parts $(ii)$ and $(iii)$ of Theorem~\ref{biratlthm} nearly holds for $wy$: instead of \eqref{tauwu}, $\tau(wy)$ takes the form
\begin{equation*}
    \pm e\Bigl(\sum\limits_{n_{\alpha}\in V_w}\frac{\e_{\a}}{n_\a}\Bigr)\prod_{n_{\a}\in V_w} \abs{n_{\a}}^{t_{\a}}\sgn (n_{\a})^{\eta_{\a}}
\end{equation*}
for some choices of sign $\e_\a=\pm 1$ for each $\a\in Inv(\pi\i)$. Step 3 simply involves letting $u_\a=\e_\a y_\a$. \par
As an example, consider again the example $w_1$ in \eqref{ex1:coords}.
\[\begin{aligned}\left(
\begin{smallarray}{cccc}
 0 & 0 & 0 & 1 \\
 0 & 1 & 0 & n_{2,4} \\
 0 & 0 & 1 & n_{3,4} \\
 1 & n_{1,2} & n_{1,3} & n_{1,4} \\
\end{smallarray}
\right) &\xrightarrow{\substack{n_{1,2} \\ \rightsquigarrow x_{1,2}}} \left(
\begin{smallarray}{cccc}
 0 & 0 & 0 & 1 \\
 0 & 1 & 0 & n_{2,4} \\
 0 & 0 & 1 & n_{3,4} \\
 1 & n_{1,2}+\frac{n_{1,4}-n_{1,3} n_{3,4}}{n_{2,4}} & n_{1,3} & n_{1,4} \\
\end{smallarray}
\right) \xrightarrow{\substack{n_{1,3} \\ \rightsquigarrow x_{1,3}}} \left(
\begin{smallarray}{cccc}
 0 & 0 & 0 & 1 \\
 0 & 1 & 0 & n_{2,4} \\
 0 & 0 & 1 & n_{3,4} \\
 1 & n_{1,2}-\frac{n_{1,3} n_{3,4}}{n_{2,4}} & n_{1,3}+\frac{n_{1,4}}{n_{3,4}} & n_{1,4} \\
\end{smallarray}
\right) \\
& \xrightarrow{x_\a \rightsquigarrow y_\a}\left(
\begin{smallarray}{cccc}
 0 & 0 & 0 & 1 \\
 0 & 1 & 0 & n_{2,4} \\
 0 & 0 & 1 & n_{2,4} n_{3,4} \\
 1 & \frac{n_{1,2} n_{1,3} n_{1,4}-n_{1,3} n_{1,4} n_{2,4}}{n_{2,4}} & \frac{n_{1,3}n_{1,4}+n_{1,4}n_{3,4}}{n_{3,4}} & n_{1,4} n_{2,4} n_{3,4} \\
\end{smallarray}
\right) \\ & \xrightarrow{y_\a \rightsquigarrow u_\a}
\left(
\begin{smallarray}{cccc}
 0 & 0 & 0 & 1 \\
 0 & 1 & 0 & n_{2,4} \\
 0 & 0 & 1 & n_{2,4} n_{3,4} \\
 1 & \frac{-n_{1,2} n_{1,3} n_{1,4}-n_{1,3} n_{1,4} n_{2,4}}{n_{2,4}} & \frac{n_{1,3}n_{1,4}+n_{1,4}n_{3,4}}{n_{3,4}} & n_{1,4} n_{2,4} n_{3,4} \\
\end{smallarray}
\right). \end{aligned}\]
In step 1, the variables in the rightmost column remain unchanged. It is straightforward to verify that the resulting entries agree with Definition~\ref{def:Rformula} of $R(n_\a)$.
\end{subsection}
\subsection{Algorithm by nonintersecting lattice paths}
This algorithm computes the map $R=R_{w}$, by following Definition~\ref{def:Rformula}. We regard $wn$ as a directed graph as in \eqref{ex1:coords}. For $n_\a\in V_w$ with $\abs{D(n_\a)}=\abs{O(n_\a)}=k$, write
\[D=D(n_\a)=\{1_{t(1)},\ldots,1_{t(k)}\}, \quad O=O(n_\a)=\{n_{\b(1)},\ldots, n_{\b(k)}\},\]
where each $1\leq t(i)\leq r$ denotes the column number and each $\b(i)$ is an index satisfying either $n_{\b(i)}\in V_w$ or $n_{\b(i)}=1_l$ for some $1\leq l\leq r$. 
\par
\underline{Step 1: Connect $O$ and $D$.} \par 
For each permutation $\sigma\in \mathfrak{S}_k$, consider the pairing $(n_{\b(i)},1_{t(\sigma(i))})$, $1\leq i\leq k$ of $O$ and $D$. If any $n_{\b(i)}$ and $1_{t(\sigma(i))}$ are not connected, let $\mathcal{P}_{\sigma}=\widetilde{\mathcal P}_{\sigma}=\emptyset$. If not, find $\mathcal{P}(\{n_{\b(i)}\}\to \{1_{t(\sigma(i))}\})$ for each $1\leq i\leq k$, and let
\[{\mathcal P}_{\sigma}=\left\{\{p_1,\ldots, p_k\} \mid \{p_i\}\in \mathcal{P}(\{n_{\b(i)}\}\to \{1_{t(\sigma(i))}\}),\  1\leq i\leq k\right\}.\]
\par
\underline{Step 2: Exclude intersections.} \par 
From ${\mathcal P}_{\sigma}$, exclude the sets of paths $\{p_1,\ldots,p_k\}$ containing intersecting paths. Any intersection can be detected simply by computing $\prod_{1\leq j\leq k}u(p_j)$, which is a monomial in $\Z[\{n_\a \vert n_\a\in V_w\}]$. An intersection occurs whenever the monomial contains a variable whose exponent is greater than $1$. Let $\widetilde{\mathcal P}_{\sigma}$ be the resulting collection consisting only of sets of disjoint paths connecting each $n_{\b(i)}$ to $1_{t(\sigma(i))}$. Let ${\mathcal P}(n_\alpha)=\bigcup_{\sigma\in \mathfrak{S}_k}\widetilde{\mathcal P}_{\sigma}$. \par

\underline{Step 3: $n_\a \rightsquigarrow u_\a$.} \par
Compute $P(n_\a)$ using ${\mathcal P}(n_\alpha)$ from step 2, following Definition~\ref{def:calPn}. Compute also $\rho(1_{t(i)})$ by Definition~\ref{rhodef} and take
$u_\alpha=(-1)^{k-1}\frac{P(n_\alpha)}{\prod\limits_{1\leq i\leq k}\rho(1_{t(i)})}$.

\section*{Statements and Declarations}
\textbf{Competing Interests}: The author declares that they have no competing interests relevant to this article. \par \noindent
\textbf{Data Availability Statement}: Data sharing not applicable to this article as no datasets were generated or analyzed during the current study.

%\nocite{*}
%\bibliography{WhitDistref}

\bibliography{WhitDistref}

\end{document}